\documentclass[12pt]{amsart}
\usepackage{amssymb,amsmath}
\usepackage{tikz}
\usepackage{tikz-cd}

\usepackage{graphicx}
\graphicspath{ {./images/} }

\def\dint{\displaystyle\int}
\def\dsum{\displaystyle\sum}

\DeclareMathOperator{\ind}{ind}

\DeclareMathOperator{\supp}{supp}
\DeclareMathOperator{\im}{im}

\newtheorem{thm}{Theorem}

\newtheorem{lemma}{Lemma}[section]
\newtheorem{Def}{Definition}[section]
\newtheorem{prop}{Proposition}[section]
\newtheorem{cor}{Corollary}[section]

\newtheorem{innercustomthm}{Theorem}
\newenvironment{cthm}[1]
  {\renewcommand\theinnercustomthm{#1}\innercustomthm}
  {\endinnercustomthm}

\begin{document}

\title{K-theoretic Virtual Fundamental Cycle of a Global Kuranishi Chart}
\author{Dun Tang}

\begin{abstract}
In this paper, we define the virtual fundamental cycle of a global Kuranishi chart as an element in the (analytic) orbispace K-homology of the virtual orbifold and verify that it defines the same invariants as those in \cite{Abouzaid23}.
\end{abstract}

\maketitle

\tableofcontents

\section{Introduction}
\

The motivation of this paper is to define permutation-equivariant quantum K-theory \cite{Lee01}\cite{perm7}\cite{perm9} for general symplectic target spaces.
The essential ingredient in this definition is the $K$-theoretical virtual fundamental cycle of moduli spaces $\mathcal{M}$ of $J$-holomorphic curves.
There are various theories characterizing these moduli spaces.
We use almost complex global Kuranishi charts \cite{Abouzaid21}\cite{Abouzaid23} in this paper.

\begin{Def}
An almost complex global Kuranishi chart $\mathcal{A}=(\mathcal{Y},\mathcal{E},\sigma)$ of $\mathcal{M}$ consists of an orbi-section $\sigma$ of a complex orbi-bundle $\mathcal{E}$ over an almost complex orbifold $\mathcal{Y}$ with boundary, such that $\sigma^{-1}(0)=\mathcal{M} \subseteq \mathcal{Y} \setminus \partial \mathcal{Y}$.
\end{Def}

In this paper, we define $K$-theoretical virtual fundamental cycles as elements in the $K$-homology of certain spaces, instead of the $K$-cohomology of cotangent spaces as in \cite{Abouzaid21}\cite{Abouzaid23}.

\subsection{The case of manifolds}
\

Let $Y$ be a compact almost complex manifold (possibly with boundary), and let $s$ be a (continuous, not necessarily smooth) section of a complex vector bundle $E$ over $Y$ whose zero locus $M = s^{-1}(0)$ lies in the interior of $Y$.
Pick Hermitian metrics on $T^\ast Y$ and $E$.

Let $\pi: T^\ast Y \to Y$ and define
\[\begin{array}{llll}
S_+ &= \wedge^{\text{even}} T^\ast Y,& S_- &= \wedge^{\text{odd}} T^\ast Y,\\
E_+ &= \wedge^{\text{even}} E, & E_-&= \wedge^{\text{odd}} E.
\end{array}\]
Given a section $\tau: X \to F$ of a Hermitian bundle $F$ over a manifold $X$, we define $\lambda_\tau: \wedge^\ast F \to \wedge^\ast F$ by $\tau(x) \wedge \cdot +i_{\tau(x)}(\cdot)$ over points $x \in X$, where $i_{\tau(x)}$ is the contraction with $\tau(x)$ via the Hermitian metric on $F$.
Form $\lambda_{\xi}: \pi^\ast S_\pm \to \pi^\ast S_\mp$ and $\lambda_s: E_\pm \to E_\mp$ corresponding to the diagonal map $\xi: T^\ast Y \to \pi^\ast T^\ast Y$ and $s:Y \to E$.

Define 
\[\begin{array}{llll}
O_+ &= (S_+ \otimes E_+) & \oplus & (S_- \otimes E_-),\\
O_- &= (S_- \otimes E_+) & \oplus & (S_+ \otimes E_-),
\end{array}\]
and
\[\alpha = \left[\begin{matrix}
\lambda_{\xi} \otimes 1 & 1 \otimes \pi^\ast \lambda_{s}\\
1 \otimes \pi^\ast \lambda_{s} & -\lambda_{\xi} \otimes 1\end{matrix}\right]: \pi^\ast O_+ \to \pi^\ast O_-.\]

Pick a vector bundle $O'$ whose direct sum with $O_+$ is trivial, and form $(\alpha': \pi^\ast O'_+ \to \pi^\ast O'_- ) = (\alpha \oplus 1: \pi^\ast(O_+ \oplus O') \to \pi^\ast (O_- \oplus O'))$.
Since $\alpha^\ast \alpha$ is multiplication by $|\xi|^2+|s(x)|^2$, $\alpha$, and hence $\alpha'$, is a bundle isomorphism outside the compact set $s^{-1}(0)$.
Under a suitable trivialization of $O'_\pm$, $\alpha'$ is the identity map near $\partial Y$.
We can also modify $\alpha$ so that it is constant along $\{(x,\nu \xi): \nu \in \mathbb{R}_+\}$ over the entire $Y$ by \cite{Atiyah68}, pages 492–493.

Let $H_\pm = L^2(Y, O'_\pm)$.
Note that $H_{\pm}$ carry the structure of $\mathbb{C} (Y)$-modules via natural multiplication by continuous $\mathbb{C}$-valued functions on the base, which can be understood as *-representations $\phi_{\pm}: \mathbb{C}(Y) \to L(H_{\pm})$ into the algebras of bounded operators on the Hilbert spaces.
Let $T_s: H_+ \to H_-$ be the order 0 pseudo-differential operator obtained by continuously extending the quantization of $\alpha'$ from compactly supported smooth sections of $O'_\pm$ to $H_\pm$.
Let $i_!: \mathbb{C}(M) \to \mathbb{C}(Y)$ be an arbitrary continuous linear operator, such that when restricted to each point, it is a continuous extension of the corresponding function.
The datum $(H_{\pm}, \phi_{\pm} \circ i_!, T_s)$ defines an element in Kasparov's K-homology group $K_0(M)$, which is independent of $i_!$ and the global chart $(Y,E,s)$.
We denote this element as $\mathcal{O}^{vir}$ and refer to it as the K-theoretic virtual fundamental class.

In particular, if $s$ is transverse to the zero section so that $M=s^{-1}(0)$ is a stably almost complex submanifold, such an extension construction allows us to define the K-theoretic fundamental class $[M] \in K_0(\Sigma^N M^+)=K_0(M)$.

\begin{thm}
$[M] = \mathcal{O}^{vir}$.
\end{thm}

Given a virtual bundle $V=V_+ \ominus V_- \in K^0(M)$, we define the pairing $\left<\mathcal{O}^{vir},V\right>$ as the difference $n_+ - n_-$ of Fredholm indices:
\[\begin{array}{ll}
n_+ &= \ind (T_s \otimes 1_+: H_+\otimes_{\mathbb{C}(M)} \mathbb{C}(V_+) \to H_-\otimes_{\mathbb{C}(M)} \mathbb{C}(V_+)); \\
n_- &= \ind (T_s \otimes 1_-: H_+ \otimes_{\mathbb{C}(M)} \mathbb{C}(V_-) \to H_- \otimes_{\mathbb{C}(M)} \mathbb{C}(V_-)),
\end{array}\]
where $\mathbb{C}(V)$ is the space of complex-valued continuous sections of the bundle $V$ and has the structure of a finitely generated $\mathbb{C}(M)$-module.

In particular, if $V = i^\ast \tilde{V}$ for $i: M \to Y$ the inclusion, and $\tilde{V} \in K^0(Y)$, then we have the following theorem.

\begin{thm}
The pairing $\left<\mathcal{O}^{vir},V\right>$ equals the index of the virtual bundle
\[O \otimes \pi^\ast \tilde{V} \in K^0(T^\ast Y),\]
where $O = O_+ \ominus O_-$.
\end{thm}

This theorem shows that our definition of invariants agrees with that in \cite{Abouzaid23}.

\subsection{Results}
\

In this paper, we define a $K$-theoretic virtual fundamental class $\mathcal{O}^{vir}$ of an almost-complex global Kuranishi chart $(\mathcal{Y},\mathcal{E},\sigma)$.
The $K$-theoretic virtual fundamental class is an element in $K_0(|I\mathcal{M}|)$, the (complex-coefficient) $K$-homology of the coarse space $|I\mathcal{M}|$ of the inertia orbi{\em space} of $\mathcal{M}$. 
We also define the pairing of $\mathcal{O}^{vir}$ with elements in the (orbispace) $K$-group of $\mathcal{M}$.

Sector-wise, $\mathcal{O}^{vir}$ is constructed from $(|\mathcal{Y}^{(g)}|, str_{(g)} \mathcal{E}|_{\mathcal{Y}^{(g)}}, str_{(g)} \circ \sigma|_{\mathcal{Y}^{(g)}})$ as in the no-isotropy case.

The pairing is the cap product composed with a certain map $p_\ast: K_0(|I\mathcal{Y}|) \to \mathbb{C}$.
\begin{itemize}
\item
The cap product of $(T: H_+ \to H_-) \in K_0(|I\mathcal{M}|)$ and $(\alpha: V_+ \to V_-) \in K^0(|I\mathcal{M}|)$ is
\[\left[\begin{matrix}T \otimes 1 & 1 \otimes C(\alpha^\ast) \\ 1 \otimes C(\alpha) & - T^\ast \otimes 1 \end{matrix}\right]: \begin{matrix} H_+ \otimes C(V_+) &&  H_- \otimes C(V_+) \\ \oplus & \to & \oplus \\ H_- \otimes C(V_-) &&  H_+ \otimes C(V_-)\end{matrix}.\]
Here, in these representatives, $H_\pm$ are Hilbert spaces, $V_\pm$ are Hermitian bundles, $C(V_\pm)$ are continuous sections of $V_\pm$, and $C(\alpha): C(V_+) \to C(V_-)$ is the induced map.
\item
For orbi{\em folds}, the map $p_\ast$ is an extension of the Fredholm index on $K$-homology classes represented by elliptic operators.
The precise definition is given in Section 3.5 via inclusion-exclusion and Kawasaki's index theorem \cite{Kawasaki81}.
Its definition is extended to virtual orbifolds in Section 4.2.
\end{itemize}
We prove the following theorems.

\begin{cthm}{1} (Orbifold version)
If $\sigma$ is transverse to the zero section, so that $\mathcal{M}=\sigma^{-1}(0)$ is a stably almost complex sub-orbifold, such an extension construction allows us to define the $K$-theoretic fundamental class $[\mathcal{M}] \in \oplus_{(g)} K_0(\Sigma^N |\mathcal{M}^{(g)}|^+)=K_0(\mathcal{M})$.
Then \[[\mathcal{M}] = \mathcal{O}^{vir}.\]
\end{cthm}

\begin{cthm}{2} (Orbifold version)
Let $i: \mathcal{M} = \sigma^{-1}(0) \to \mathcal{Y}$ be the inclusion.
Given a virtual bundle $\mathcal{V} \in K^0(\mathcal{Y})$, the pairing of $\mathcal{O}^{vir}$ with $i^\ast \mathcal{V}$ equals the invariants defined by $\mathcal{V}$ in \cite{Abouzaid23}.
\end{cthm}

\section*{Acknowledgments}
\

The author thanks Professor Alexander Givental for drawing his attention to quantum K-theory for symplectic manifolds, suggesting related problems, and providing patience and guidance throughout the process.

\section{Orbifold K-theory}
\

Let $\mathcal{Y}=Y/G$ be a compact orbifold with boundary (unless otherwise mentioned), which is the quotient of a compact manifold $Y$ by a compact Lie group $G$ (\cite{Pardon20} Corollary 1.3).
We use cursive letters for objects on orbifolds and printed-style letters for $G$-equivariant and $G$-invariant objects on manifolds.
We use $\varepsilon$ for the grading operator on various $\mathbb{Z}_2$-graded objects, i.e., the operator that acts as $(-1)^a$ on the subspace of grading $a$.
We use $\mathbb{C}(\mathcal{Y})$ for the $\mathbb{C}^\ast$-algebra of complex-valued continuous functions (under the $\sup$-norm) on $\mathcal{Y}$.

\subsection{(Compactly supported) K-group}
\

We recall the relative K-group $K^0(\mathcal{Y},\partial \mathcal{Y})$ of the orbi{\em space} $\mathcal{Y}=Y/G$.
Note that in this paper the orbispace is either an orbifold (the base space of the global Kuranishi chart) or a virtual orbifold (the zero locus of the section in the global Kuranishi chart).
In the latter case, $\partial \mathcal{Y} = \emptyset$.

For an equivariant bundle map $\alpha: E_+ \to E_-$ between complex $G$-bundles over $Y$, we define the support of $\alpha$ as the (closed) set of points in $Y$ where $\alpha$ is not an isomorphism of the representations of the isotropy groups.
The set $\mathcal{C}_G(Y,\partial{Y})$ of all bundle maps $\alpha$ supported away from $\partial Y$ is a semigroup under direct sum (defined as $\alpha^0 \oplus \alpha^1: E^0_+ \oplus E^1_+ \to E^0_- \oplus E^1_-$ for $\alpha^i: E^i_+ \to E^i_-$).
The zero element is the identity map between 0-dimensional trivial bundles over $Y$.

An element $\alpha: E_+ \to E_-$ in $\mathcal{C}_G(Y,\partial{Y})$ is called degenerate if it has empty support.
Two elements $\alpha^0: E^0_+ \to E^0_-$ and $\alpha^1: E^1_+ \to E^1_-$ in $\mathcal{C}_G(Y,\partial{Y})$ are homotopic if there exists an element $\tilde{\alpha}: \tilde{E}_+ \to \tilde{E}_-$ in $C_G(Y \times [0,1], (\partial Y) \times [0,1])$ that restricts to $\alpha^0$ and $\alpha^1$ at $Y \times \{0\}$ and $Y \times \{1\}$, respectively.
Note that the pair $(Y \times [0,1], (\partial Y) \times [0,1])$ equivariantly retracts to $(Y,\partial Y)$; thus we can take $\tilde{E}_\pm$ to be the pull-back of a certain bundle over $(Y,\partial Y)$, i.e., independent of $t \in [0,1]$.

We define an equivalence relation on $\mathcal{C}_G(Y,\partial{Y})$ generated by identifying degenerate elements with $0$ and by identifying homotopic elements.
We define the $G$-equivariant $K$-group $K_G^0(Y,\partial{Y})$ of $Y$ as the set of equivalence classes of this relation under direct sum.

In the literature, the $K$-group $K_G^0(Y, \partial Y)$ is sometimes defined as follows.
We define the absolute $K$-groups $K^0_G$ of a space as the Grothendieck group associated to the semigroup (under direct sum) of isomorphism classes of finite-dimensional complex $G$-bundles over that space.
A map between spaces induces a pull-back of bundles, which is a homomorphism between the corresponding absolute $K$-groups.
We define the reduced $K$-group $\tilde{K}_G^0(Y/\partial Y, pt)$ as the kernel of the pull-back map $K^0_G(Y/\partial Y) \to K^0_G(pt)$ induced by the inclusion $pt \to Y/\partial Y$ of the base point, and define $K_G^0(Y, \partial Y) = \tilde{K}_G^0(Y/\partial Y, pt)$.

Note that, by definition, each element in $\tilde{K}_G^0(Y/\partial Y, pt)$ is represented by an element for which $E_+|_U = E_-|_U = U \times_G V$, for some representation $V$ of $G$, on a $G$-invariant neighborhood $U$ of $\partial Y$.
This definition coincides with our first definition using compactly supported bundle homomorphisms (\cite{Segal68} Proposition 3.1, see also \cite{Atiyah68} page 490).
An isomorphism is explicitly constructed as follows.
Given $E_+ \ominus E_- \in \tilde{K}_G^0(Y/\partial Y, pt)$, let $\alpha: E_+ \to E_-$ be an arbitrary map that restricts to the identity on $U$.
The pair $(E_\pm, \alpha)$ represents an element in $K^0_G(Y,\partial Y)$.
Conversely, given a bundle map $\alpha: E_+ \to E_-$, we pick a bundle $E$ over $Y$, and a $G$-invariant neighborhood $U$ of $\partial Y$, such that $(E_+ \oplus E)|_U$ is of the form $U \times_G V$.
Then $\alpha \oplus id: E_+ \oplus E \to E_- \oplus E$ represents an element in $\tilde{K}_G^0(Y/\partial Y, pt)$.

For simplicity, we define $K^0_G(Y,\partial{Y})$ as pairs $(E, \alpha)$ where $E$ is a $\mathbb{Z}_2$-graded Hermitian $G$-bundle and $\alpha: E \to E$ is a self-adjoint, degree $1$, compactly supported bundle automorphism, modulo degenerate elements and homotopy.
Indeed, $(E_+, E_-, \alpha_+)$ in the first definition corresponds to $E = E_+ \oplus E_-$ and the map
\[\alpha = \left[\begin{matrix} 0 & \alpha_+^\ast \\ \alpha_+ & 0\end{matrix}\right]\]
Here, we identify $(E_\pm)^\ast$ with $E_\pm$ by choosing a Hermitian metric on $E_\pm$.
Different choices of Hermitian metrics on $E_\pm$ yield homotopic representatives of the $K$-class, and hence the same element in the $K$-group.

The orbispace K group is defined as $K^0(Y/G,\partial{Y}/G) = K_G^0(Y,\partial{Y})$.

\subsection{Twisted sectors, inertia orbispace, and fake K-group}
\

Given a quotient orbispace $\mathcal{Y}=Y/G$, we define the inertia orbispace
\[I\mathcal{Y}=\left(\coprod_{g \in G} Y^g \times \{g\}\right)/G.\]
Here $h\in G$ acts on $(x,g)$ as $h \circ (x,g)=(hx,hgh^{-1})$, and $Y^g$ is the $g$-fixed locus in $Y$.
We define $\mathcal{Y}^{(g)}$ (or simply $\mathcal{Y}^g$, when no confusion arises) as the component in $I\mathcal{Y}$ labeled by the conjugacy class $(g)$.
We call the $\mathcal{Y}^{(g)}$ the sectors, and those with $g \neq 1$ the twisted sectors.

More generally, given $\mathbf{g}=(g_1,\cdots,g_k) \in G^k$, we use $Y^{\mathbf{g}}$ for the locus fixed by $g_1, \cdots, g_k$.
We define 
\[I^k\mathcal{Y} = \left(\coprod_{\mathbf{g}} \mathcal{Y}^\mathbf{g} \times \{\mathbf{g}\}\right)/G,\]
where $h\in G$ acts on $(x,\mathbf{g})$ by $h \circ (x,\mathbf{g})=(hx,h\mathbf{g}h^{-1})$.
Define $\mathcal{Y}^{(\mathbf{g})}$ (or $\mathcal{Y}^{\mathbf{g}}$ when no confusion arises) as the sector in $I^k\mathcal{Y}$ corresponding to the conjugacy class $(\mathbf{g})$.
We call $\mathcal{Y}^{(\mathbf{g})}$ a $k$-multi-sector.

By `fake' theories of orbispaces, we mean theories in which we only remember the topological space structure on $\mathcal{Y}$ (and forget the isotropy groups at points in $\mathcal{Y}$).
As an example, define the fake orbispace K-group $K_\mathbb{Z}^{0,fake}(\mathcal{Y}, \partial \mathcal{Y})$ as classes represented by $(\alpha = \alpha^\ast: \mathcal{E} \to \mathcal{E}) \in K^0(\mathcal{Y}, \partial \mathcal{Y})$, with the isotropy groups at points in $\mathcal{Y}$ acting trivially on the fibers of $\mathcal{E}$.
Define 
\[K^{0,fake}(\mathcal{Y}, \partial \mathcal{Y}) = K_\mathbb{Z}^{0,fake}(\mathcal{Y}, \partial \mathcal{Y}) \otimes \mathbb{C}.\]
Note that the super-trace of a $global$ isotropy $(g) \subseteq G$ defines a map 
\[K^0(\mathcal{Y}, \partial \mathcal{Y}) \otimes \mathbb{C} \to K^{0,fake}(\mathcal{Y}, \partial \mathcal{Y}).\]
Note also that by the eigen-bundle decomposition, there is an isomorphism 
\[K^0(\mathcal{Y}, \partial \mathcal{Y}) \otimes \mathbb{C} \to K^{0,fake}(I\mathcal{Y}, \partial I \mathcal{Y}),\]
given by sending $\alpha$ to the element with $str_{(g)}\, i_{(g)}^\ast \alpha$ in the twisted sector labeled by $(g)$, where $i_{(g)}: \mathcal{Y}^{(g)} \to \mathcal{Y}$ is the inclusion.

\subsection{K-homology}
\

We consider Kasparov $\mathbb{C}(\mathcal{Y})$-modules $(H,\phi,T)$, where by definition $H$ is a $\mathbb{Z}_2$-graded complex Hilbert space and $T$ is a degree $1$ self-adjoint automorphism of $H$. The map $\phi: \mathbb{C}(\mathcal{Y}) \to L(H)$ is a *-homomorphism from $\mathbb{C}(\mathcal{Y})$ to the space $L(H)$ of bounded linear maps on $H$, such that each $\phi(f)$ is of degree $0$ and almost commutes with $T$ in the sense that $[T,\phi(f)]$ is compact in $L(H)$ \cite{Kasparov75}. The set of Kasparov $\mathbb{C}(\mathcal{Y})$-modules is an abelian semi-group under direct sum: $(H,\phi,T) \oplus (H',\phi',T') = (H \oplus H',\phi \oplus \phi',T \oplus T')$, with the zero element defined by $H=0$, $\phi(f)=0$, and $T=0$.

We say that $(H,\phi,T)$ is unitarily equivalent to $(H',\phi',T')$ if there is a degree $0$ unitary map $U: H \to H'$ that takes $(\phi, T)$ to $(\phi', T')$. We say that $(H,\phi,T_0)$ is homotopic to $(H,\phi,T_1)$ if there exists a continuous family of operators $\{T_t\}_{t\in[0,1]}$ connecting $T_0$ and $T_1$ such that $(H,\phi,T_t)$ is a Kasparov module for all $t\in[0,1]$. We say that $(H,\phi,T)$ is degenerate if $T$ is invertible and commutes with $\phi(f)$ for all $f \in \mathbb{C}(\mathcal{Y})$.

We identify Kasparov modules that differ by a degenerate Kasparov $\mathbb{C}(\mathcal{Y})$-module and also identify unitarily equivalent and homotopic Kasparov modules to obtain the fake $K$-homology (semi-)group $K^{fake}_0(\mathcal{Y})$.
Indeed, $K^{fake}_0(\mathcal{Y})$ is a group \cite{Kasparov75}. The inverse of $(H, \phi, T)$ is $(H',\phi', -T)$, where $H'$ is $H$ with the grading shifted by $1$, and $\phi'=\phi\varepsilon$ with $\varepsilon$ acting as $(-1)^{a}$ on components of grading $a$ (Proposition 1 in \cite{Kasparov75}).

We define $K_0(\mathcal{Y})=K^{fake}_0(I \mathcal{Y})$. Note that a map $f: \mathcal{Y} \to \mathcal{Y}'$ induces 
\[\begin{array}{lll}
f_\ast:& K^{fake}_0(\mathcal{Y}) &\to K^{fake}_0(\mathcal{Y}'),\\
& (H, \phi, T) &\to (H, \phi \circ f^\ast, T).
\end{array}\]

\subsection{Embedding relatively compact sets into compact orbifolds}
\

Starting from this subsection, {\bf we assume $\mathcal{Y}$ is an orbi{\em fold}}. In this subsection, we prove a technical proposition.

\begin{prop}
Any compact subset $\mathcal{U} \subseteq \mathcal{Y} \setminus \partial \mathcal{Y}$ in a Riemannian orbifold $\mathcal{Y}$ has a neighborhood that embeds openly into a closed orbifold $\mathcal{W}$.
\end{prop}

We first prove the following lemma.

\begin{lemma}
The subset $\mathcal{U}$ is contained in the interior of a compact sub-orbifold $\mathcal{X} \subseteq \mathcal{Y} \setminus \partial \mathcal{Y}$ whose boundary $\partial \mathcal{X}$ is (an orbifold) transverse to twisted strata.
\end{lemma}

\begin{proof}
Assume without loss of generality that $\mathcal{U}$ is connected. By the compactness of $\mathcal{U}$, there are finitely many geodesic balls $B(x_i,r_i)$ that cover $\mathcal{U}$, such that the exponential map at each $x_i$ is a diffeomorphism between $B(x_i,2r_i)$ and its pre-image. We start with an arbitrary compact codimension-$1$ orbifold without boundary that meets some $B(x_i,r_i)$ and modify this orbifold on all the $B(x_i,r_i)$'s to obtain a codimension-$1$ orbifold that bounds a compact region $\mathcal{X}$ which contains all the $B(x_i,r_i)$.

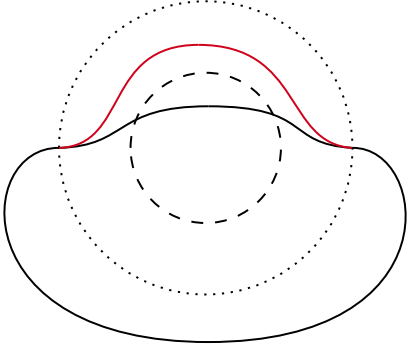
\tikzset{every picture/.style={line width=0.75pt}}
\begin{tikzpicture}[x=0.75pt,y=0.75pt,yscale=-1,xscale=1]
\draw  [dash pattern={on 0.84pt off 2.51pt}] (32,81) .. controls (32,40.13) and (65.13,7) .. (106,7) .. controls (146.87,7) and (180,40.13) .. (180,81) .. controls (180,121.87) and (146.87,155) .. (106,155) .. controls (65.13,155) and (32,121.87) .. (32,81) -- cycle ;
\draw  [dash pattern={on 4.5pt off 4.5pt}] (68.08,81) .. controls (68.08,60.06) and (85.06,43.08) .. (106,43.08) .. controls (126.94,43.08) and (143.92,60.06) .. (143.92,81) .. controls (143.92,101.94) and (126.94,118.92) .. (106,118.92) .. controls (85.06,118.92) and (68.08,101.94) .. (68.08,81) -- cycle ;
\draw [color={rgb, 255:red, 0; green, 0; blue, 0 }  ,draw opacity=1 ]   (32,81) .. controls (67.33,80) and (59.33,60) .. (107.33,60) ;
\draw [color={rgb, 255:red, 0; green, 0; blue, 0 }  ,draw opacity=1 ]   (107.33,60) .. controls (160.33,60) and (145.33,79) .. (180,81) ;
\draw [color={rgb, 255:red, 208; green, 2; blue, 27 }  ,draw opacity=1 ]   (32,81) .. controls (67.33,80) and (54.33,29) .. (102.33,29) ;
\draw [color={rgb, 255:red, 208; green, 2; blue, 27 }  ,draw opacity=1 ]   (102.33,29) .. controls (155.33,29) and (145.33,79) .. (180,81) ;
\draw    (32,81) .. controls (-10.4,80.49) and (-16.67,179) .. (107.33,179) .. controls (231.33,179) and (218.63,80.51) .. (180,81) ;
\end{tikzpicture}

In each step of the modification, we pick a geodesic ball $B(x_i,r_i)$ that meets the codimension-$1$ orbifold and “push” the sub-orbifold (black) out (red) of this ball within $B(x_i,2r_i)$, while fixing points near the boundary of $B(x_i,2r_i)$ (by pushing or pulling along the radii of the corresponding geodesics emerging from the $x_i$'s). We apply this construction until the resulting codimension-$1$ orbifold $\partial \mathcal{X}$ does not meet any $B(x_i,r_i)$. In this case, $\partial \mathcal{X}$ bounds a compact region $\mathcal{X}$ that contains all the $B(x_i,r_i)$, and hence contains $\mathcal{U}$.

Moreover, since $\mathcal{U}$ is relatively compact, there are only finitely many twisted strata. Under the assumption that the space $\mathcal{Y}$ is connected, each stratum is either the entire $\mathcal{Y}$ or a sub-orbifold of strictly lower dimension. We can make $\partial \mathcal{X}$ transverse to these strata by a perturbation.
\end{proof}

\begin{proof}
(Proof of Proposition 2.1)

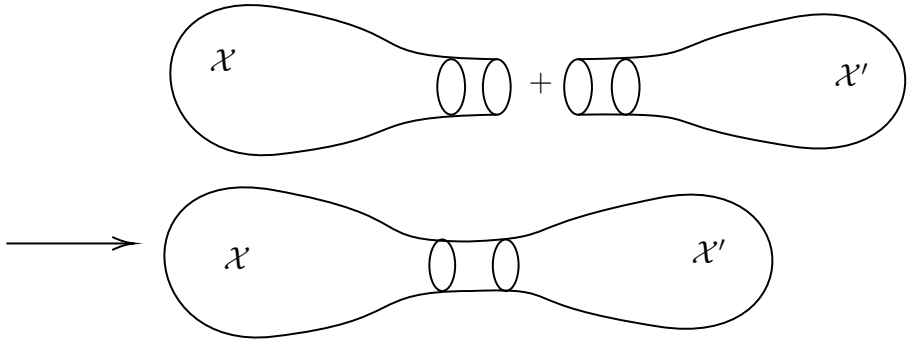
\tikzset{every picture/.style={line width=0.75pt}}
\begin{tikzpicture}[x=0.75pt,y=0.75pt,yscale=-1,xscale=1]
\draw    (296.33,87) .. controls (369.33,86) and (327.33,91) .. (404.33,103) .. controls (481.33,115) and (479.33,23) .. (404.33,38) .. controls (329.33,53) and (369.33,58) .. (296.33,59) ;
\draw    (255.33,87) .. controls (183.33,88) and (214.33,99) .. (145.33,107) .. controls (76.33,115) and (68.7,19.34) .. (145.33,33) .. controls (221.97,46.66) and (182.33,58) .. (255.33,59) ;
\draw   (248.33,73) .. controls (248.33,65.27) and (251.47,59) .. (255.33,59) .. controls (259.2,59) and (262.33,65.27) .. (262.33,73) .. controls (262.33,80.73) and (259.2,87) .. (255.33,87) .. controls (251.47,87) and (248.33,80.73) .. (248.33,73) -- cycle ;
\draw   (289.33,73) .. controls (289.33,65.27) and (292.47,59) .. (296.33,59) .. controls (300.2,59) and (303.33,65.27) .. (303.33,73) .. controls (303.33,80.73) and (300.2,87) .. (296.33,87) .. controls (292.47,87) and (289.33,80.73) .. (289.33,73) -- cycle ;
\draw   (225.33,73) .. controls (225.33,65.27) and (228.47,59) .. (232.33,59) .. controls (236.2,59) and (239.33,65.27) .. (239.33,73) .. controls (239.33,80.73) and (236.2,87) .. (232.33,87) .. controls (228.47,87) and (225.33,80.73) .. (225.33,73) -- cycle ;
\draw   (313.33,73) .. controls (313.33,65.27) and (316.47,59) .. (320.33,59) .. controls (324.2,59) and (327.33,65.27) .. (327.33,73) .. controls (327.33,80.73) and (324.2,87) .. (320.33,87) .. controls (316.47,87) and (313.33,80.73) .. (313.33,73) -- cycle ;
\draw    (8,152) -- (70.33,152) ;
\draw [shift={(72.33,152)}, rotate = 180] [color={rgb, 255:red, 0; green, 0; blue, 0 }  ][line width=0.75]    (10.93,-3.29) .. controls (6.95,-1.4) and (3.31,-0.3) .. (0,0) .. controls (3.31,0.3) and (6.95,1.4) .. (10.93,3.29)   ;
\draw    (242.33,176) .. controls (188.33,176) and (211.33,191) .. (142.33,199) .. controls (73.33,207) and (65.7,111.34) .. (142.33,125) .. controls (218.97,138.66) and (187.33,151) .. (241.33,151) ;
\draw    (242.33,176) .. controls (292.33,174) and (260.33,182) .. (337.33,194) .. controls (414.33,206) and (412.33,114) .. (337.33,129) .. controls (262.33,144) and (290.33,149) .. (241.33,151) ;
\draw   (221.33,163) .. controls (221.33,155.82) and (224.24,150) .. (227.83,150) .. controls (231.42,150) and (234.33,155.82) .. (234.33,163) .. controls (234.33,170.18) and (231.42,176) .. (227.83,176) .. controls (224.24,176) and (221.33,170.18) .. (221.33,163) -- cycle ;
\draw   (253.33,163) .. controls (253.33,155.82) and (256.24,150) .. (259.83,150) .. controls (263.42,150) and (266.33,155.82) .. (266.33,163) .. controls (266.33,170.18) and (263.42,176) .. (259.83,176) .. controls (256.24,176) and (253.33,170.18) .. (253.33,163) -- cycle ;
\draw (270,64.4) node [anchor=north west][inner sep=0.75pt]    {$+$};
\draw (116,152.4) node [anchor=north west][inner sep=0.75pt]    {$\mathcal{X}$};
\draw (352,148.4) node [anchor=north west][inner sep=0.75pt]    {$\mathcal{X} '$};
\draw (109,52.4) node [anchor=north west][inner sep=0.75pt]    {$\mathcal{X}$};
\draw (424,59.4) node [anchor=north west][inner sep=0.75pt]    {$\mathcal{X} '$};
\end{tikzpicture}

There is a collar neighborhood $\varphi: \partial \mathcal{X} \times (0,1) \to \mathcal{X}$ of $\partial \mathcal{X}$ in $\mathcal{X}$. We take two copies $\mathcal{X} = \mathcal{X}'$, with collar neighborhoods $\varphi$ and $\varphi'$, and define $\mathcal{W}$ as the “doubling”
\[(\mathcal{X} \coprod \mathcal{X}')/(\varphi(x,t) \sim \varphi'(x,1-t)).\]
Then $\mathcal{W}$ is closed, and there is an embedding $\mathcal{U} \hookrightarrow \mathcal{X} \hookrightarrow \mathcal{W}$.
\end{proof}

\begin{cor}
Any compact subset $U \subseteq Y \setminus \partial Y$ in a Riemannian $G$-manifold $Y$ has a $G$-invariant neighborhood that equivariantly open embeds into a closed, almost free $G$-manifold $W$.
\end{cor}

\subsection{The `choose-an-operator' map}
\

For a compact $G$-manifold $Y$ (with boundary), define $T^\ast_G Y$ as the set of covectors on $Y$ that annihilate all vectors tangent to $G$-orbits. 
Equip $Y$ with a $G$-invariant Hermitian structure and let $B^\ast_GY$ be the unit ball in $T^\ast_GY$. In this subsection, we construct a `choose-an-operator' map $\mu=\mu_{G, Y}: K_G^{0, inv}(B_G^\ast Y, \partial B_G^\ast Y) \to K^{fake}_0(Y/G)$, where $K_G^{0, inv}$ stands for those elements in $K_G^0$ represented by virtual bundles corresponding to the trivial representation of $G$. The orbifold `choose-an-operator' map $\mu_{Y/G}$ is defined sector-wise via $\mu_{G,Y}$. (The `choose-an-operator' map when $G$ is the trivial group is constructed in, e.g., \cite{Baum16}.)

We first define certain “good” representatives of elements in $K_G^{0, inv}(B_G^\ast Y, \partial B_G^\ast Y)$, which we call {\em symbolic representatives}, that serve as the symbols of certain pseudo-differential operators representing K-homology classes.

If any two automorphisms of the same bundle agree outside a compact set, they represent the same element in the $K$-group. Thus, to describe an element in $K_G^{0, inv}(B_G^\ast Y, \partial B_G^\ast Y)$, we only need to define the bundle involved and the automorphism near the boundary. Let $\pi_Y: B_G^\ast Y \to Y$, and let $p: \pi_Y^\ast E \to \pi_Y^\ast E$ be a bundle isomorphism defined outside a compact subset—which we refer to as $\supp p$—in $Y \setminus \partial Y \subseteq B_G^\ast Y$.

\ \ \ \ \ \ \ \ \ \ \ \ \ \ \ \ \ \ \ \ \ \ \ \ \ \ \ \ \ \
\tikzset{every picture/.style={line width=0.75pt}}
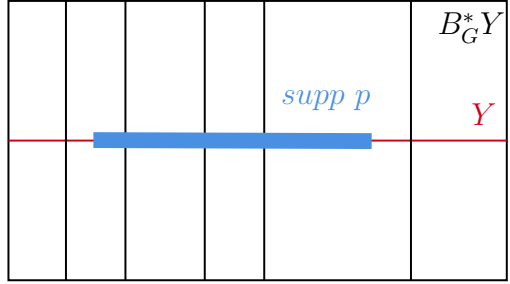
\begin{tikzpicture}[x=0.75pt,y=0.75pt,yscale=-1,xscale=1]
\draw   (100,24) -- (351.33,24) -- (351.33,165) -- (100,165) -- cycle ;
\draw [color={rgb, 255:red, 208; green, 2; blue, 27 }  ,draw opacity=1 ]   (100,94.5) -- (351.33,94.5) ;
\draw    (129,24) -- (129,165) ;
\draw    (303,24) -- (303,165) ;
\draw    (159,24) -- (159,165) ;
\draw    (199,24) -- (199,165) ;
\draw    (229,24) -- (229,165) ;
\draw [color={rgb, 255:red, 74; green, 144; blue, 226 }  ,draw opacity=1 ][line width=6]    (142.83,94.25) -- (283.17,94.75) ;
\draw (332,74.4) node [anchor=north west][inner sep=0.75pt]  [color={rgb, 255:red, 208; green, 2; blue, 27 }  ,opacity=1 ]  {$Y$};
\draw (315,28.4) node [anchor=north west][inner sep=0.75pt]    {$B_{G}^{\ast } Y$};
\draw (236,67.4) node [anchor=north west][inner sep=0.75pt]  [color={rgb, 255:red, 74; green, 144; blue, 226 }  ,opacity=1 ]  {$supp\ p$};
\draw (98,175) node [anchor=north west][inner sep=0.75pt]   [align=left] {Symbolic representative $\displaystyle p$:\\Constant along each vertical line segment in $\displaystyle B_{G}^{\ast } Y\setminus supp\ p$};
\end{tikzpicture}

A symbolic representative $p$ supported on $\supp p \subseteq B_G^\ast(Y)$ is an automorphism of $\pi_Y^\ast E$ outside $E|_{\supp p}$, with
\begin{itemize}
\item
$E$ a $G$-bundle over $Y$ that corresponds to the trivial representation and is trivial outside $\supp p$,
\item
$p(y,\xi) = p(y,\nu\xi)$ for all $(y,\xi) \in B_G^\ast Y \setminus Y$ and for all $\nu \in \mathbb{R}_+$.
\end{itemize}
Symbolic representatives, modulo the relations of homotopy and isomorphism, form a group that is isomorphic to $K_G^{0, inv}(B_G^\ast Y, \partial B_G^\ast Y)$ (cf. \cite{Atiyah68} page 491-493, also \cite{Atiyah74} page 15). Moreover, we can take this symbolic representative to be self-adjoint.

By Corollary 2.1, we can equivariantly embed a $G$-invariant neighborhood of $\supp p$ into a closed manifold $W$ with an almost free $G$-action. Let $E_\pm$ be the positive and negative parts of the $\mathbb{Z}_2$-grading on $E$, and set $p_+ = p|_{E_+}$. There exist sections of $E_\pm|_{Y \setminus \supp p}$ that form a basis on each fiber, under which the matrix of $p_+$ is the identity. Thus, we can extend $p$ and $E$ trivially to $p_W$ and $E_W$ on $W$. Let $P_W: C^\infty(W, E_W) \to C^\infty(W, E_W)$ be the quantization of the (self-adjoint) symbolic representative $p_W$. Then $P_W$ is a self-adjoint pseudo-differential operator of order $0$ and is of degree $1$ under the $\mathbb{Z}_2$-grading. If $s \in C^\infty(W, E_W)$ is supported in some $U \supseteq \supp p$, then $P_Ws$ is also supported in $U$. If $s \in C^\infty(W, E_W)$ is supported in $W \setminus \supp p$, then $P_Ws = \psi s$ for some invertible map $\psi$ on $W \setminus \supp p$. Let $H = L^2(U, E|_U)^G$ be the space of $G$-invariant $L^2$-sections of $E$ restricted to $U$, which is identified with a subspace of $L^2(W,E_W)$ by extending each section by $0$ outside $U$. Since $P_W: C^\infty(W,E_W) \to C^\infty(W,E_W)$ is elliptic, it uniquely extends in a continuous way to $L^2$-sections. Let $T$ be the restriction of this continuous ($G$-equivariant) extension to the subspace $H$. Let $\phi$ denote the $\mathbb{C}(Y/G)$-module structure on $H$, given by multiplication (restricted to $U$) by a continuous $G$-invariant function. We define 
\[\mu_{G,Y} (\pi_Y^\ast E,p) = (H,\phi,T).\]

\begin{lemma}
The element $(H,\phi,T)$ represents an element in $K^{fake}_0(\mathcal{Y})$.
\end{lemma}

\begin{proof}
Following \cite{Atiyah74}, we show that $T$ is Fredholm by proving that the $G$-invariant parts $(\ker P_W)^G = (\ker P_W^\ast)^G$ are finite-dimensional. Pick a bi-invariant Riemannian metric on $G$, and let $v_1,\cdots, v_k$ be first-order differential operators acting on sections of $E$ corresponding to an orthonormal basis of the Lie algebra of $G$. Consider the bundle map $\pi_W^\ast E \to \pi_W^\ast E$ given by 
\[\delta_G:(x,\xi,e) \to \left(x,\xi,\frac{\sum_i\left<v_i,\xi\right>^2}{1+\sum_i\left<v_i,\xi\right>^2} e\right).\]
The quantization of $\delta_G$ yields an order $0$ pseudo-differential operator whose kernel contains the space of $G$-invariant functions as a closed subspace. Now, the operator $A=(P_W,\Delta_G)$ has an injective symbol, so the symbol of $A^\ast A$ is invertible. Thus, the operator $A^\ast A = P_W^\ast P_W + \Delta^\ast_G \Delta_G$ is elliptic with finite-dimensional kernel. Since $(\ker P_W)^G=(\ker P_W^\ast)^G$ is contained in the kernel of $A^\ast A$, it is also finite-dimensional. Hence, $T$ is Fredholm.

Moreover, for an arbitrary $f$, the order $0$ symbol of the commutator $[P_W|_U,\phi(f)]$ is zero, so $[T,\phi(f)]$ is compact. Thus, $\mu_{G,Y} (\pi_Y^\ast E,p)$ represents an element in $K^{fake}_0(Y)$.
\end{proof}

\begin{prop}
The operator $\mu_{G,Y}$ is well-defined as a map 
\[K_G^{0,inv}(B_G^\ast Y, \partial B_G^\ast Y) \to K_0^{fake}(Y).\]
\end{prop}

\begin{proof}
We temporarily denote the `choose-an-operator' map with the choices $U,W,p$ by $\mu_{G,Y}^{U,W}(p)$. We need to prove:
\begin{enumerate}
\item
$\mu_{G,Y}^{U,W}(p)$ is independent of the symbolic representative $p$.
\item
$\mu_{G,Y}^{U,W}([p])$ is independent of the choices of $U$ and $W$.
\end{enumerate}

For fixed $U$ and $W$, let $P_W$ and $P_W'$ be two pseudo-differential operators with symbolic representatives satisfying $[p_W]=[p_W']$. If $p_W$ and $p_W'$ are homotopic (in particular, if $p_W=p_W'$ but $P_W \neq P_W'$), i.e., if their symbols are connected by a homotopy of degree $1$ automorphisms of the same bundle $\pi^\ast_W E$, then the Hilbert spaces involved are the same and the operators are connected by a homotopy (given by the quantization of the symbols), hence representing the same class in $K$-homology. If $p_W$ is degenerate in $K_G^{0, inv}(B_G^\ast W,\partial B_G^\ast W)$, then it is an isomorphism between the positive and negative parts in the $\mathbb{Z}_2$-grading of the ambient bundle with empty support. Under the identification given by this isomorphism, the operator $T$ is the identity, and thus $\mu_{G, Y}(p_W)$ is degenerate.

For fixed $U$, the Hilbert spaces are naturally identified, and different choices of $W$ yield the same operator $T$. For $\supp p \subseteq U' \subseteq U$, let $(E,p) \in K_G^{0,inv}(B_G^\ast Y, \partial B_G^\ast Y)$ be a symbolic representative. The corresponding $K$-homology element defined using $U'$ versus $U$ differs by the image of the degenerate element $(E,p)|_{U \setminus U'}$ under the choose-an-operator map. Thus, they define the same element in $K_0^{fake}(Y)$. Since any two general compact sets are contained in their union, the defined element in $K_0^{fake}(Y)$ is independent of the choices of $U$ and $W$.
\end{proof}

We have the following direct corollary.

\begin{cor}
The choose-an-operator map is compatible with open embeddings.
\end{cor}

The cotangent bundle of $\mathcal{Y}=Y/G$ is $T^\ast \mathcal{Y} = (T^\ast_GY)/G$. Denote its unit ball bundle by $B^\ast \mathcal{Y}$ (with respect to a Hermitian metric). Observe that the sector $(B^\ast \mathcal{Y})^{(g)} = B^\ast (\mathcal{Y})^{(g)}$, and that the supertrace $str_g$ maps $K^{0,fake}(B^\ast \mathcal{Y}^{(g)}, \partial B^\ast \mathcal{Y}^{(g)})$ to $K_G^{0,inv}(B_G^\ast Y^{(g)}, \partial B_G^\ast Y^{(g)}) \otimes \mathbb{C}$. We define the ‘genuine’ choose-an-operator map $\mu$ sector-wise as the composition of the invariant choose-an-operator map with the supertraces. The orbifold choose-an-operator map is independent of the presentation of the orbifold.

\section{The pairing}
\

\subsection{The Cap product}
\

In this Subsection, we define the K-theoretic cap product on $\mathcal{Y}$. Given $(H, \phi, T) \in K^{fake}_0(\mathcal{Y})$ and $(\mathcal{E},\alpha) \in K^{0,fake}(\mathcal{Y},\partial \mathcal{Y})$, we construct an element $(\tilde{H}, \tilde{\phi}, \tilde{T}) \in K^{fake}_0(\mathcal{Y})$ as the cap product $\alpha \cap T$, following \cite{Kasparov73}.

We first describe the Hilbert $\mathbb{C}(\mathcal{Y})$-module structure $(\tilde{H},\tilde{\phi})$. For a complex orbi-bundle $\mathcal{E}$, let $C(\mathcal{E})$ be continuous sections of $\mathcal{E}$. Then $C(\mathcal{E})$ is a finitely generated projective module over $\mathbb{C}(\mathcal{Y})$. Thus, there is an idempotent matrix $A$ of size $n$, with entries in $\mathbb{C}(\mathcal{Y})$, whose range is $C(\mathcal{E})$. It is viewed as a map $A: H^n \to H^n$ via $\phi: \mathbb{C}(\mathcal{Y}) \to L(H)$. Set $\tilde{H} = \im A$ and let $\tilde{\phi}$ be the natural $\mathbb{C}(\mathcal{Y})$-module structure on $\tilde{H}$.

To describe the Fredholm operator $\tilde{T}: \tilde{H} \to \tilde{H}$, let $C(\alpha)$ be the $\mathbb{C}(\mathcal{Y})$-module homomorphism $C(\mathcal{E}) \to C(\mathcal{E})$ induced by $\alpha: \mathcal{E} \to \mathcal{E}$. Let $T \otimes' 1$ and $1 \otimes C(\alpha)$ be the induced maps on $\im A$, and define
\[\tilde{T} = T \otimes' 1  -\varepsilon 1 \otimes C(\alpha).\]
Recall that in the introduction, we defined \(\varepsilon\) as the operator on \(\mathbb{Z}_2\)-graded objects acting as \((-1)^a\) on the direct summand of degree \(a\).

We give the following remarks.
\begin{enumerate}
\item 
We have $\tilde{H} = H \otimes_{\mathbb{C}(\mathcal{Y})} C(\mathcal{E})$ as vector spaces, but different $A$'s define different inner products on this vector space. Nonetheless, all such inner products define the same topology (which is all we need by Lemma 3.2).
\item 
While $1 \otimes C(\alpha)$ is the actual tensor product, we emphasize that the ‘tensor product’ $T \otimes 1$ does not make sense since the (pseudo-differential) operator $T$ is not $\mathbb{C}(\mathcal{Y})$-linear. Therefore, we use the notation $\otimes'$ and need to introduce the projection $A$ in our construction.
\end{enumerate}

\begin{lemma}
The element $(\tilde{H}, \tilde{\phi},\tilde{T})$ represents an element in $K^{fake}_0(\mathcal{Y})$.
\end{lemma}
\begin{proof}
Observe that $T\varepsilon + \varepsilon T=0$, $\varepsilon^2=1$, and that $C(\alpha)$ commutes with $\tilde{\phi}(f)$, so we have
\[\begin{array}{ll}
[\tilde{T},\tilde{\phi}(f)] &= [T,\phi(f)] \otimes' 1;\\[1mm]
\tilde{T}^2 &= T^2 \otimes' 1 + 1 \otimes C(\alpha)^2.
\end{array}\]
Hence, $\tilde{T}$ almost commutes with $\tilde{\phi}(f)$ and $\tilde{T}^\ast=\tilde{T}$ has a finite-dimensional kernel.
\end{proof}

\begin{lemma}
$(H, \phi, T) \in K^{fake}_0(\mathcal{Y})$ depends only on $T$, $\phi$, and the topology on $H$ (and is independent of the inner products on $H$).
\end{lemma}

\begin{proof}
Let $(H, \phi, T)$ be a Kasparov module. Let $H' = H$ as $\mathbb{C}(\mathcal{Y})$-modules and topological spaces, each endowed with a certain inner product. Then
\[
\tilde{T}_t = \left[\begin{matrix}T \cos (\frac\pi2t) & \sin (\frac\pi2t) I \\[1mm] \sin (\frac\pi2t) I & -T \cos (\frac\pi2t) \end{matrix}\right]
\]
defines a homotopy from $(H, \phi, T) \ominus (H', \phi, T)$ to a degenerate Kasparov module. Thus, $(H', \phi, T) = (H, \phi, T)$ in $K^{fake}_0(\mathcal{Y})$.
\end{proof}

\begin{prop}
The (fake) cap product is a map 
\[K_0^{fake}(\mathcal{Y}) \otimes K^{0,fake}(\mathcal{Y}, \partial \mathcal{Y}) \to K_0^{fake}(\mathcal{Y}).\]
\end{prop}

\begin{proof}
Given $(H,\phi,T) \in K_0^{fake}(\mathcal{Y})$ and $(\mathcal{E},\alpha) \in K^{0,fake}(\mathcal{Y}, \partial \mathcal{Y})$, we show that
\begin{enumerate}
\item
The cap product $(\tilde{H}, \tilde{\phi},\tilde{T})$ is independent of the choice of $A$ (and $n$).
\item
The cap product is independent of the representatives $(H,\phi, T)$ and $(\mathcal{E},\alpha)$ chosen.
\end{enumerate}

Let $A, A'$ be idempotent matrices of sizes $n$ and $n'$ with $\mathbb{C}(\mathcal{Y})$ coefficients, whose images are identified via a $\mathbb{C}(\mathcal{Y})$-module homomorphism $\iota$ that is also an isometry. Then the composition

\begin{tikzcd}
\tilde{\iota}:(\mathbb{C}(\mathcal{Y}))^n \arrow[r, "A"] & \im A \arrow[r, "{(\iota,0)}"] & \im A' \oplus \im (I-A') = (\mathbb{C}(\mathcal{Y}))^{n'}
\end{tikzcd}

\noindent is a lift of $\iota$ as a module homomorphism, thus realizing the map $\tilde{H} \to \tilde{H}'$ as a matrix transformation, and so this map is continuous. Likewise, the map $\tilde{H}' \to \tilde{H}$ induced by $\iota^{-1}$ is also continuous, so the topologies on $\tilde{H}$ and $\tilde{H}'$ are identical. Moreover, $\tilde{\iota}$ intertwines $\phi(f)$'s and $T$'s up to a compact operator, so different $A$'s give unitarily equivalent cap products.

Different representatives of $K$-homology differ by a sequence of unitary equivalences, homotopies, or direct sums with degenerate objects. Similarly, different representatives of the same element in the $K$-group differ by a sequence of isomorphisms, homotopies, or direct sums with degenerate objects. The cap product, in turn, is also connected by a sequence of unitary equivalences, homotopies, or direct sums with degenerate objects. Thus, the cap product 
\[\cap: K^{fake}_0(\mathcal{Y}) \otimes K^{0,fake}(\mathcal{Y}, \partial \mathcal{Y}) \to K^{fake}_0(\mathcal{Y})\]
is well defined.
\end{proof}

\begin{lemma}
The (fake) cap product is compatible with embeddings.
\end{lemma}

\begin{proof}
Given an inclusion of compact sets $\mathcal{Y} \subseteq \mathcal{Y}'$, we view $C(\mathcal{E}|_{\mathcal{Y}})$ as the range of a $\mathbb{C}(\mathcal{Y})$-matrix. Indeed, by restricting the coefficients of this matrix (which are functions on $\mathcal{Y}'$) to $\mathcal{Y}$, we obtain a matrix whose range gives $C(\mathcal{E}|_{\mathcal{Y}})$. Computing under this specific choice of basis shows that each inclusion of $\mathcal{Y}$ induces a map on cap products.
\end{proof}

The genuine cap product 
\[\cap: K_0(\mathcal{Y}) \otimes K^0(\mathcal{Y}, \partial \mathcal{Y}) \to K_0(\mathcal{Y})\]
is defined sector-wise.

\subsection{$K^{0}(\mathcal{Y}, \partial \mathcal{Y})$-module structures}
\

In this subsection, we show that the $K^{0}(\mathcal{Y}, \partial \mathcal{Y})$ module structures on $K^{0}(B^\ast \mathcal{Y}, \partial B^\ast \mathcal{Y})$ and on $K_0(\mathcal{Y})$ are intertwined by $\mu_\mathcal{Y}$.

Let $\pi: B^\ast \mathcal{Y} \to \mathcal{Y}$.
Given elements $(\mathcal{E},\alpha) \in K^{0,fake}(\mathcal{Y}, \partial \mathcal{Y})$ and $(\mathcal{F},\beta) \in K^{0,fake}(B^\ast \mathcal{Y}, \partial B^\ast \mathcal{Y})$, we define (see \cite{Atiyah68} page 490)
\[
\alpha \boxtimes \beta = \beta \otimes 1 - \varepsilon \otimes \pi^\ast \alpha: \mathcal{F} \otimes \pi^\ast \mathcal{E} \to \mathcal{F} \otimes \pi^\ast \mathcal{E}.
\]

This element is supported in the (compact) intersection of $\supp \alpha \subseteq \mathcal{Y} \subseteq B^\ast \mathcal{Y}$ and $\supp \beta \subseteq B^\ast \mathcal{Y}$, making $K^{0,fake}(B^\ast \mathcal{Y}, \partial B^\ast \mathcal{Y})$ a $K^{0,fake}(\mathcal{Y}, \partial \mathcal{Y})$-module. On the other hand, $K^{fake}_0(\mathcal{Y})$ carries a $K^{0,fake}(\mathcal{Y}, \partial \mathcal{Y})$-module structure given by the cap product. We use $\hat{\otimes}$ for the completion of the algebraic tensor product.

\begin{lemma}
Let $\mathcal{E},\mathcal{F}$ be Hermitian bundles with trivial isotropy group action over a compact orbifold $\mathcal{Y}$. Then 
\[L^2(\mathcal{E}) \hat{\otimes}_{\mathbb{C}(\mathcal{Y})} C(\mathcal{F}) \cong L^2(\mathcal{E} \otimes_{\mathbb{C}} \mathcal{F})\]
as $\mathbb{C}(\mathcal{Y})$-modules and as topological spaces.
\end{lemma}

\begin{proof}
If $\mathcal{F} = \underline{n}$, the $n$-dimensional trivial bundle over $\mathcal{Y}$, then $C(\underline{n}) = (\mathbb{C}(\mathcal{Y}))^n$ induces an isomorphism $L^2(\mathcal{E}) \hat{\otimes}_{\mathbb{C}(\mathcal{Y})} C(\underline{n}) = \Bigl(L^2(\mathcal{E}) \hat{\otimes}_{\mathbb{C}(\mathcal{Y})} \mathbb{C}(\mathcal{Y})\Bigr)^n = \Bigl(L^2(\mathcal{E} \otimes \underline{1})\Bigr)^n = L^2(\mathcal{E} \otimes \underline{n})$, as $\mathbb{C}(\mathcal{Y})$-modules and topological spaces.

For general $\mathcal{F}$, there is an $n$ and a complex vector bundle $\mathcal{F}'$ such that $\mathcal{F} \oplus \mathcal{F}' = \underline{n}$ by the compactness of $\mathcal{Y}$. The natural map $L^2(\mathcal{E}) \hat{\otimes}_{\mathbb{C}(\mathcal{Y})} C(\mathcal{F}) \to L^2(\mathcal{E} \otimes \underline{n})$ is injective with image contained in $L^2(\mathcal{E} \otimes \mathcal{F})$, since $C(\underline{n}) = C(\mathcal{F}) \oplus C(\mathcal{F}')$.
Thus, $L^2(\mathcal{E}) \hat{\otimes}_{\mathbb{C}(\mathcal{Y})} C(\mathcal{F})$ may be viewed as a sub-module of $L^2(\mathcal{E} \otimes \mathcal{F})$. Similarly, $L^2(\mathcal{E}) \hat{\otimes}_{\mathbb{C}(\mathcal{Y})} C(\mathcal{F}') \subseteq L^2(\mathcal{E} \otimes \mathcal{F}')$.
Now, it follows from 
\[
L^2(\mathcal{E}) \hat{\otimes}_{\mathbb{C}(\mathcal{Y})} C(\mathcal{F}) \oplus L^2(\mathcal{E}) \hat{\otimes}_{\mathbb{C}(\mathcal{Y})} C(\mathcal{F}') = L^2(\mathcal{E} \otimes \underline{n}) = L^2(\mathcal{E} \otimes \mathcal{F}) \oplus L^2(\mathcal{E} \otimes \mathcal{F}')
\]
that 
\[
L^2(\mathcal{E}) \hat{\otimes}_{\mathbb{C}(\mathcal{Y})} C(\mathcal{F}) = L^2(\mathcal{E} \otimes \mathcal{F}).
\]
It is also easily seen from the proof that the topologies are preserved.
\end{proof}

\begin{prop}
The following diagram commutes.

\begin{tikzcd}
K^0(B^\ast \mathcal{Y}, \partial B^\ast \mathcal{Y}) \otimes K^0(\mathcal{Y}, \partial \mathcal{Y}) \arrow[d, "\boxtimes"] \arrow[r, "\mu \otimes 1"] & K_0(\mathcal{Y}) \otimes K^0(\mathcal{Y}, \partial \mathcal{Y}) \arrow[d, "\cap"] \\
K^0(B^\ast \mathcal{Y}, \partial B^\ast \mathcal{Y}) \arrow[r, "\mu"]                                                 & K_0(\mathcal{Y})                                 
\end{tikzcd}
\end{prop}

\begin{proof}
We only need to prove the sector-wise commutativity

\begin{tikzcd}
K^{0,fake}(B^\ast \mathcal{Y}, \partial B^\ast \mathcal{Y}) \otimes K^{0,fake}(\mathcal{Y}, \partial \mathcal{Y}) \arrow[d, "\boxtimes"] \arrow[r, "\mu^{fake} \otimes 1"] & K^{fake}_0(\mathcal{Y}) \otimes K^{0,fake}(\mathcal{Y}, \partial \mathcal{Y}) \arrow[d, "\cap"] \\
K^{0,fake}(B^\ast \mathcal{Y}, \partial B^\ast \mathcal{Y}) \arrow[r, "\mu^{fake}"]                                                 & K^{fake}_0(\mathcal{Y}) 
\end{tikzcd}

To do this, let $\beta: \mathcal{F} \to \mathcal{F}$ be an element in $K^{0,fake}(\mathcal{Y})$. Let 
\[P: L^2(\mathcal{E}) \to L^2(\mathcal{E})\]
be a pseudo-differential operator of order $0$ that represents 
\[\alpha: \mathcal{E} \to \mathcal{E}.\]
Let $\mathcal{U}$ be a relatively compact open sub-manifold such that $\alpha$ is an isomorphism and such that $P$ is given by multiplication by a function outside $\mathcal{U}$. We have the following results.

\[\begin{array}{lll}
\beta \boxtimes \sigma(P) &= \beta \otimes 1 -\varepsilon \otimes \sigma(P):& \pi^\ast \mathcal{E} \otimes \pi^\ast \mathcal{F} \to \pi^\ast \mathcal{E} \otimes \pi^\ast \mathcal{F}, \\
P \cap \beta &= P \otimes 1 - \varepsilon \otimes C(\beta):& L^2(\mathcal{E}) \hat{\otimes}_{\mathbb{C}(\mathcal{Y})} C(\mathcal{F}) \to L^2(\mathcal{E}) \hat{\otimes}_{\mathbb{C}(\mathcal{Y})} C(\mathcal{F}).
\end{array}\]
The operator
\[
[P] \cap \beta = [P] \otimes 1 - \varepsilon \otimes C(\beta): L^2(\mathcal{E} \otimes \mathcal{F}) \to L^2(\mathcal{E} \otimes \mathcal{F})
\]
represents $\beta \boxtimes \sigma(P)$ because $\sigma(P \otimes 1) = \sigma(P) \otimes 1$. It also represents $[P] \cap \beta$ by the previous lemma. Therefore, $\mu([\sigma(P)] \boxtimes [\beta]) = [P] \cap [\beta]$, and so the diagram commutes.
\end{proof}

\subsection{K-homological maps $i^!$ and $i_\ast$ and $K$-theoretical map $i_!$}
\

Let $i: \mathcal{Z} \to \mathcal{Y}$ be a closed subset (not necessarily an orbifold) of $\mathcal{Y}$ whose boundary is contained in that of $\mathcal{Y}$.

Recall that 
\[i_\ast^{fake} (H,\phi,T) = (H, \phi \circ i^\ast, T).\]

Conversely, we define 
\[i^{fake, !}(H,\phi,T) = (H, \phi \circ \tilde{i}_!, T),\]
where $\tilde{i}_!: \mathbb{C}(\mathcal{Z}) \to \mathbb{C}(\mathcal{Y})$ is continuous and linear and is an extension for each function in $\mathbb{C}(\mathcal{Z})$.

\begin{lemma}
Such an extension $\tilde{i}_!$ exists.
\end{lemma}

\begin{proof}
One construction is as follows. We extend the real and imaginary parts separately in the same way so that the result is complex linear. We only need to extend functions on $\mathcal{Z}$ with maximum absolute value $1$. We modify the proof of Tietze's extension theorem by incorporating a distance $d$ on $\mathcal{Y}$. For each $f$, in the first step, we set $f_0=f$,
\[\begin{array}{ll}
\mathcal{Z}_{k,+} &= \overline{\{x \in \mathcal{Z}: f_k(x)>(\frac13)^k\}},\\[1mm]
\mathcal{Z}_{k,-} &= \overline{\{x \in \mathcal{Z}: f_k(x)<-(\frac13)^k\}},
\end{array}\]
and define 
\[g_k(x) = \frac{-d(x,\mathcal{Z}_{k,+}) + d(x,\mathcal{Z}_{k,-})}{d(x,\mathcal{Z}_{k,+}) + d(x,\mathcal{Z}_{k,-})} \Bigl(\frac13\Bigr)^k;\]
\[f_{k+1} = f_k - g_k.\]
The extension is given by $\sum_{k>0} g_k$.
\end{proof}

\begin{lemma}
$i^{fake,!}$ is independent of the extension $\tilde{i}_!$. 
\end{lemma}

\begin{proof}
Let $(H,\phi_{1,2}, T)$ be the representatives we constructed from different extensions. Then 
\[
\tilde{T}_t = \left[\begin{matrix}T \cos (\frac\pi2t) & \sin (\frac\pi2t) I \\[1mm] \sin (\frac\pi2t) I & -T \cos (\frac\pi2t) \end{matrix}\right]
\]
defines a homotopy from $(H,\phi_1,T) \ominus (H,\phi_2,T)$ to a degenerate element.
\end{proof}

By definition, we have the following lemma.

\begin{lemma}
The following diagram commutes for any inclusion $i: \mathcal{Z} \to \mathcal{Y}$.

\begin{tikzcd}
{K^{0,fake}(\mathcal{Y}) \otimes K_0^{fake}(\mathcal{Y})} \arrow[d, "i^\ast_{fake} \otimes i^!_{fake}"] \arrow[r, "\cap"] & K_0^{fake}(\mathcal{Y}) \arrow[d, "i^!_{fake}"] \\
{K^{0,fake}(\mathcal{Z}) \otimes K_0^{fake}(\mathcal{Z})} \arrow[r, "\cap"]                                               & K_0^{fake}(\mathcal{Z})                        
\end{tikzcd}
\end{lemma}

From now on, assume that $\mathcal{Z}$ is a suborbifold. We recall the wrong‐way $K$-cohomological push-forward 
\[i_!^{fake}: K^{0,fake}(B^\ast \mathcal{Z}, \partial B^\ast \mathcal{Z}) \to K^{0,fake}(B^\ast \mathcal{N}, \partial B^\ast \mathcal{N})\]
defined in \cite{Atiyah68}. 
Let $\mathcal{N}_{\mathbb{C}} = B(\pi_\mathcal{Z}^\ast \mathcal{N} \otimes \mathbb{C})$, which is identified with $B^\ast \mathcal{N} \to B^\ast \mathcal{Z}$ by identifying the real part with the fibers of $\mathcal{N} \to \mathcal{Z}$ and the imaginary part with the corresponding cotangent vectors. We define the K-theoretical push-forward in the special case of $\mathcal{Y}=\mathcal{N}$ as
\[i_!^{fake}:\alpha \to (\pi^\ast \alpha) \otimes (str_1 \wedge^\ast \mathcal{N}_\mathbb{C}^\ast).\]
The push-forward $i_!^{fake}$ in the general case is the composition of $i_!^{fake}$ for a tubular neighborhood and the map $K^{0,fake}(B^\ast \mathcal{N}, \partial B^\ast \mathcal{N}) \to K^{0,fake}(B^\ast \mathcal{Y}, \partial B^\ast \mathcal{Y})$ given by extending a virtual bundle trivially outside $\mathcal{N}$. Moreover, $i_!^{fake}$ is independent of the tubular neighborhood chosen \cite{Atiyah74}.

We define the genuine $i_!$, $i_\ast$, and $i^!$ sector‐wise. It is straightforward from the definition that
\[(ij)_\ast = i_\ast j_\ast, \ \ \ \ (ij)^! = j^!i^!, \ \ \ \ (ij)_! = i_!j_!.\]

\subsection{Compatibility}
\

\subsubsection{Statement of results}
\

We will prove the following proposition in this subsection.

\begin{prop}

The following diagrams commute.

\begin{enumerate}
\item
\begin{tikzcd}
K^0(B^\ast \mathcal{Z}, \partial B^\ast \mathcal{Z}) \arrow[r, "i_!"] \arrow[d, "\mu_\mathcal{Z}"] & K^0(B^\ast \mathcal{Y}, \partial B^\ast \mathcal{Y}) \arrow[d, "\mu_\mathcal{Y}"] \\
K_0(\mathcal{Z})                 & K_0(\mathcal{Y})   \arrow[l, "i^!"]  
\end{tikzcd}
\item
\begin{tikzcd}
K^0(B^\ast \mathcal{Z}, \partial B^\ast \mathcal{Z}) \arrow[r, "i_!"] \arrow[d, "\mu_\mathcal{Z}"] & K^0(B^\ast \mathcal{Y}, \partial B^\ast \mathcal{Y}) \arrow[d, "\mu_\mathcal{Y}"] \\
K_0(\mathcal{Z})  \arrow[r, "i_\ast"]              & K_0(\mathcal{Y}) 
\end{tikzcd}
\end{enumerate}
\end{prop}

\subsubsection{A technical lemma}
\

Let $\pi: \mathcal{N} \to \mathcal{Z}$ be a closed tubular neighborhood of $\mathcal{Z}$.
Define various projections by the following diagram.

\begin{tikzcd}
B^\ast \mathcal{N} \arrow[r, "\tilde{\pi}"] \arrow[d, "\pi_\mathcal{N}"] & B^\ast \mathcal{Z} \arrow[d, "\pi_\mathcal{Z}"] \\ 
\mathcal{N}    \arrow[r, "\pi"] & \mathcal{Z}  
\end{tikzcd}

Our goal in this subsection is to prove

\begin{lemma}
The following diagram commutes.

\begin{tikzcd}
K^{0,fake}(B^\ast \mathcal{Z}, \partial B^\ast \mathcal{Z}) \arrow[r, "i_!^{fake}"] \arrow[d, "\mu_\mathcal{Z}"] & K^{0,fake}(B^\ast \mathcal{N}, \partial B^\ast \mathcal{N}) \arrow[d, "\mu_\mathcal{N}"] \\ 
K^{fake}_0(\mathcal{Z})  & K^{fake}_0(\mathcal{N})   \arrow[l, "\pi_\ast"]  
\end{tikzcd}
\end{lemma}

We remark that composing with the index map $K^{fake}_0(\mathcal{X}) \to \mathbb{C}$ given by the (virtual) Fredholm index of $T_+: H_+ \to H_-$ we know that $i_!^{fake}$ preserves indices.
Thus this lemma could be viewed as a strengthening of corresponding theorems in index theory (e.g., Theorem 4.3 in \cite{Atiyah74} restricted to orbifolds).

We first fix some notations.
Fix a symbolic representative 
\[(\pi_{\mathcal{Z}}^\ast \mathcal{E}, \alpha) \in K^{0,fake}(B^\ast \mathcal{Z}, \partial B^\ast \mathcal{Z}).\]
We define $\beta$ and $\tilde{D}$, and a specific $\tilde{T}$ (and $T$) that fit into the following diagram. Note that the ‘product’ $\otimes'$ is as defined in Section 3.1 and is not the genuine tensor product, as $\tilde{T}$ and $\tilde{D}$ are not $\mathbb{C}(\mathcal{N})$-linear. Nonetheless, $\otimes'$ is the genuine tensor product $\otimes$ after pushing forward to $K_0^{fake}(\mathcal{Z})$ under $\pi_\ast$.

\begin{tikzcd}
(\pi_\mathcal{Z}^\ast \mathcal{E}, \alpha) \arrow[r, "i_!^{fake}"] \arrow[d, "\mu_\mathcal{Z}"] & (\tilde{\pi}^\ast \pi^\ast_\mathcal{Z} \mathcal{E} \otimes str_1 \wedge^\ast \mathcal{N}_{\mathbb{C}}^\ast, \alpha \otimes 1 -\varepsilon 1 \otimes \beta) \arrow[d, "\mu_\mathcal{N}"] \\ 
(L^2(\mathcal{Z}, \mathcal{E}), T)  & \Bigl(L^2\bigl(\mathcal{N}, \pi^\ast (\mathcal{E} \otimes str_1 (\wedge^\ast \mathcal{N}^\ast \otimes \mathbb{C}))\bigr), \tilde{T}\otimes' 1 - \varepsilon 1 \otimes' \tilde{D}\Bigr)
\end{tikzcd}

Let $\beta$ be a symbolic representative of $\lambda_{\xi}: \wedge^\ast \mathcal{N}_{\mathbb{C}} \to \wedge^\ast \mathcal{N}_{\mathbb{C}}$ (as defined in the introduction) in $K^{0,fake}(\mathcal{N})$ and let $\tilde{D}$ be its quantization.
Then $\tilde{D}$ is a self-adjoint 0th order pseudo-differential operator over $\mathcal{N}$ that acts in the fiber direction of $\mathcal{N} \to \mathcal{Z}$.

We write $\mathcal{Z} = Z/G$ and $\mathcal{N} = N/G$ and pick an equivariant, order $0$ pseudo-differential operator $P$ on $Z$ whose symbol represents $\alpha$.
Recall that in Section 2.5, we defined $T$ as (the average of) $P$ restricted to $G$-invariant $L^2$-sections. To define $\tilde{T}$, we pick $\mathcal{N}'$ such that $\mathcal{N} \oplus \mathcal{N}' = Z \times_G V$ for $V$ (the unit ball in) some $G$-representation.
We construct $\tilde{P}$ as the order $0$ pseudo-differential operator on $Z \times V$ given by $P$ on each slice $Z \times \{v\}$, and $\tilde{P}^G$ as the average of $\tilde{P}$ over $G$.
We define $\tilde{T}$ by requiring that its action on a section of $\pi^\ast \mathcal{E}$ over $\mathcal{N}$ equals the restriction to $\mathcal{N}$ of $\tilde{P}_G$ acting on the pull-back of that section along $\mathcal{N} \oplus \mathcal{N}' \to \mathcal{N}$.

\begin{lemma}
For an element $(\mathcal{H}, \phi, \mathcal{T}) \in K^{fake}_0(\mathcal{Z})$, let $\Delta \in L(\mathcal{H})$ be a degree $1$ self-adjoint operator that commutes with $\phi$ and $\mathcal{T}$. Then 
\[
(\mathcal{H}, \phi, \mathcal{T} - \varepsilon \Delta) = (\ker \Delta, \phi|_{\ker \Delta}, \mathcal{T}).
\]
\end{lemma}

\begin{proof}
The operators $\mathcal{T}$ and $\Delta$ commute, so $\ker \Delta$ and $\im \Delta$ are both invariant under $\mathcal{T} - \varepsilon \Delta$ and under $\phi(f)$. Thus, $(\ker \Delta, \phi|_{\ker \Delta}, \mathcal{T})$ and $(\im \Delta, \phi|_{\im \Delta}, \mathcal{T} - \varepsilon \Delta)$ both represent elements in $K_0^{fake}(\mathcal{Z})$, and 
\[
(\ker \Delta, \phi|_{\ker \Delta}, \mathcal{T}) \oplus (\im \Delta, \phi|_{\im \Delta}, \mathcal{T} - \varepsilon \Delta) = (\mathcal{H}, \phi, \mathcal{T} - \varepsilon \Delta).
\]
Now, $(\im \Delta, \phi|_{\im \Delta}, \mathcal{T} - \varepsilon \Delta)$ is homotopic to the degenerate element $(\im \Delta, \phi|_{\im \Delta}, - \varepsilon \Delta)$ by the line segment connecting them, so 
\[
(\ker \Delta, \phi|_{\ker \Delta}, \mathcal{T}) = (\mathcal{H}, \phi, \mathcal{T} - \varepsilon \Delta).
\]
\end{proof}

\begin{proof} (Proof of Lemma 3.8)

Observe that the image under $\pi_\ast$ of $1 \otimes \tilde{D}$ commutes with the images of $\tilde{T} \otimes 1$ and $\phi$, so by the previous lemma we know that 
\[
\pi_\ast\mu_\mathcal{N} i_!^{fake}(\mathcal{E},\alpha) = (\ker (1 \otimes \tilde{D}),\, \tilde{T}\otimes 1|_{\ker (1 \otimes \tilde{D})}).
\]
Let $\{s_i\}$ be a $\mathbb{C}(\mathcal{Z})$-basis of $C(\mathcal{N})$, and form $\beta_i, \tilde{D}_i$ from the $s_i$ as we form $\beta, \tilde{D}$. Then $1 \otimes \tilde{D}_i$ commutes with $\tilde{T} \otimes 1$ and with $\phi$. If $(V,\tilde{T} \otimes 1) \in K_0^{fake}(\mathcal{Z})$, then the self-adjoint operator $\tilde{T} \otimes 1$ almost commutes with the $\mathbb{C}(\mathcal{Z})$-action, and so represents the same element as $(V,\tilde{T} \otimes 1-\varepsilon \tilde{D}_i)$.
Thus, again by the previous lemma, we can restrict $V$ to $\ker \tilde{D}_i$. Hence, 
\[
\pi_\ast\mu_\mathcal{N} i_!^{fake}(\mathcal{E},\alpha) = \Bigl(\bigcap_i \ker (1 \otimes \tilde{D}_i), \tilde{T} \otimes 1\Bigr).
\]
On each fiber of $\mathcal{N} \to \mathcal{Z}$, the equations given by $\tilde{D}_i=0$ precisely determine a section over $\mathcal{N}$ from its initial values in $\mathcal{Z}$ (\cite{Atiyah68}). Thus the common kernel of the $\tilde{D}_i$ is isomorphic to $(L^2(\mathcal{Z}, \mathcal{E}),\tilde{T} \otimes 1) = (L^2(\mathcal{E}), T\otimes 1)$, and the proposition follows.
\end{proof}

\subsubsection{Proof of Proposition 3.3}
\

Note that the statement on twisted sectors is identical to that on the non-twisted sector (with $\mathcal{Z}$, $\mathcal{N}$, and $\mathcal{Y}$ replaced by the corresponding twisted sectors in these spaces).
It suffices to prove the statement sector-wise.

\begin{cor}
The following diagram commutes.

\begin{tikzcd}
K^{0,fake}(B^\ast \mathcal{Z}, \partial B^\ast \mathcal{Z}) \arrow[r, "i_!^{fake}"] \arrow[d, "\mu_\mathcal{Z}"] & K^{0,fake}(B^\ast \mathcal{N}, \partial B^\ast \mathcal{N}) \arrow[d, "\mu_\mathcal{N}"] \\
K^{fake}_0(\mathcal{Z})        \arrow[r, "i_\ast^{fake}"]            & K^{fake}_0(\mathcal{N})  
\end{tikzcd}

\end{cor}

\begin{proof}
Observe that the contraction $\pi: \mathcal{N} \to \mathcal{Z}$ composed with $i^{fake}: \mathcal{Z} \hookrightarrow \mathcal{N}$ is homotopic to $id: \mathcal{N} \to \mathcal{N}$, thus inducing a homotopy $\pi^\ast (i^{fake})^\ast \cong id$ on $\mathbb{C}(\mathcal{N})$. By \cite{Kasparov81} Theorem 3 in Section 3, we know that $i^{fake}_\ast \pi_\ast = id$ on $K_0^{fake}(\mathcal{N})$. On the other hand, it follows from $\pi i^{fake} = id_\mathcal{Z}$ that $\pi_\ast i^{fake}_\ast = id$ on $K_0^{fake}(\mathcal{Z})$. So $\pi_\ast$ is the inverse of $i^{fake}_\ast$, and the corollary follows.
\end{proof}

\begin{lemma}
Let $j: \mathcal{N} \to \mathcal{Y}$ be the inclusion of the closure of an open set.
\begin{enumerate}
\item
$j^!j_\ast$ is the identity map $K_0^{fake}(\mathcal{N}) \to K_0^{fake}(\mathcal{N})$.
\item
If $\mathcal{Y}$ deformation retracts to $\mathcal{N}$, then $j^!$ and $j_\ast$ are inverse to each other.
\end{enumerate}
\end{lemma}

\begin{proof}
Part (1) follows from the definition. If furthermore $\mathcal{Y}$ deformation retracts to $\mathcal{N}$, then $j_\ast$ is an isomorphism by \cite{Kasparov81} Theorem 3 in Section 3. Thus, $j_\ast = (j^!)^{-1}$ by part (1).
\end{proof}

\begin{lemma}
The following diagrams commute for each open embedding $\mathcal{N} \hookrightarrow \mathcal{Y}$.

\begin{tikzcd}
K^{0,fake}(B^\ast \mathcal{N}, \partial B^\ast \mathcal{N}) \arrow[r,"j_\ast"] \arrow[d, "\mu_\mathcal{N}"] & K^{0,fake}(B^\ast \mathcal{Y}, \partial B^\ast \mathcal{Y}) \arrow[d, "\mu_\mathcal{Y}"] \\
K^{fake}_0(\mathcal{N}) \arrow[r, "j_\ast"]                & K^{fake}_0(\mathcal{Y})               
\end{tikzcd}

\begin{tikzcd}
K^{0,fake}(B^\ast \mathcal{N}, \partial B^\ast \mathcal{N}) \arrow[r,"j_\ast"] \arrow[d, "\mu_\mathcal{N}"] & K^{0,fake}(B^\ast \mathcal{Y}, \partial B^\ast \mathcal{Y}) \arrow[d, "\mu_\mathcal{Y}"] \\
K^{fake}_0(\mathcal{N})                & K^{fake}_0(\mathcal{Y})     \arrow[l,"j^!"]           
\end{tikzcd}
\end{lemma}

\begin{proof}
The commutativity of the first diagram follows from Corollary 2.2. The commutativity of the second diagram follows from that of the first diagram and the fact that $j^!j_\ast$ is the identity map.
\end{proof}

Proposition 3.3 follows from Lemmas 3.10, 3.11, and Corollary 3.1.

\subsection{Push-forward from inertia orbifold to point}
\

Consider sequences 
\[
(g,\, g_1G_1,\, \cdots,\, g_nG_n)
\]
of cosets in \(G\), where \(G_{i+1}\) is the group generated by the previous coset \(g_iG_i\) and where \(G_1\) is the group generated by \(g\). Let \(\mathcal{G}\) denote the set of such sequences in which, for each \(i\), the element \(g_i\) is not contained in \(G_i\). Given a sequence \(\mathbf{g} = (g, g_1, \cdots, g_n)\), we construct a sequence of cosets by letting \(G_i\) be the group generated by \(g, g_1, \cdots, g_i\). In other words, the sequence \(\mathbf{g}\) defines an element in \(\mathcal{G}\) if and only if each \(g_i\) is not contained in the group generated by the preceding elements.

Given \(\mathbf{g}\), let \(\mathcal{Y}^{\mathbf{g}}\) be the \((n+1)\)-multi-sector corresponding to \(\mathbf{g}\), and let 
\[
i_\mathbf{g}:\mathcal{Y}^{\mathbf{g}} \to \mathcal{Y}
\]
denote the inclusion. Let \(n=|\mathbf{g}|\) and let \(m(\mathbf{g})\) denote the order of the isotropy group of \(\mathcal{Y}^{\mathbf{g}}\). For each \(i\), let \(\mathcal{N}_{g_i}\) be the normal bundle of \(\mathcal{Y}^{(g,g_1,\cdots, g_i)}\) in \(\mathcal{Y}^{(g=g_0, g_1,\cdots, g_{i-1})}\), and in particular, let \(\mathcal{N}_g\) be the normal bundle of \(\mathcal{Y}^{(g)}\) in \(\mathcal{Y}\). Define
\[
e^{-1}_\mathbf{g} = \Bigl(\prod_{i=0}^n str_{g_i}\, \wedge^\ast\Bigl(\mathcal{N}^\ast_{g_i}\otimes \mathbb{C}\Bigr)\Bigr)^{-1}.
\]
These objects depend only on the element in \(\mathcal{G}\) represented by \(\mathbf{g}\) and are independent of the particular representative \(\mathbf{g}\).

We define the fake push-forward \(p_\ast^{fake}\) of an element \((H,\phi,T)\) in the fake \(K\)-homology to be the Fredholm index of the operator 
\[
T_+: H_+ \to H_-,
\]
where \(H_+\) and \(H_-\) denote the positive and negative parts of \(H\) under the \(\mathbb{Z}_2\)-grading, and where \(T_+ = T|_{H_+}\).

\begin{Def}
The orbifold “push-forward to point” \(p_\ast(\varphi)\) of 
\[
\{\varphi^{(g)} = (H^{(g)}, \phi^{(g)}, T^{(g)})\} \in K_0(\mathcal{Y})
\]
is defined by
\[
p_\ast(\varphi) = p_\ast^{fake} \!\left(\sum_{\mathbf{g} \in \mathcal{G}} \frac{(-1)^{|\mathbf{g}|+\dim \mathcal{Y}^{(g)}}}{m(\mathbf{g})} \Bigl(e_\mathbf{g}^{-1}\, i_\mathbf{g}^! \varphi^{(g)}\Bigr)\right).
\]
\end{Def}

Since points in the compact orbifold \(\mathcal{Y}\) have finite isotropy, there is a uniform bound \(N\) on the orders of the isotropy groups, and \(n\)-sectors \(\mathcal{Y}^{\mathbf{g}}\) become empty when \(n>N\). 
Thus, this is a finite sum.

\begin{prop} 
The following diagram commutes, where \(\mathbb{Z} \to \mathbb{C}\) is the natural inclusion and \(\ind (\cdots)_+\) maps an operator \(u\) to the index of its restriction \(u_+\).
\[
\begin{tikzcd}
K^0(B^\ast\mathcal{Y}, \partial B^\ast \mathcal{Y}) \arrow[r, "\mu"] \arrow[d, "\ind (\cdots)_+"] & K_0(\mathcal{Y}) \arrow[d, "p_\ast"] \\
\mathbb{Z} \arrow[r]                                                       & \mathbb{C}                       
\end{tikzcd}
\]
\end{prop}

\begin{proof}
Restricting or extending a symbol outside its support preserves the analytic index by the excision property (Theorem 3.7 in \cite{Atiyah74}). 
Taking into account Proposition 2.1, it suffices to prove the proposition when \(\mathcal{Y}=\mathcal{W}\) is closed.

Given \(u \in K^0(B^\ast \mathcal{Y}, \partial B^\ast \mathcal{Y})\), define 
\[
u_\mathbf{g} = e_\mathbf{g}^{-1}\, i_{\mathbf{g}}^\ast (str_g\, u).
\]
We compute \(\ind u_+\) by Kawasaki's index theorem:
\[
\ind u_+ = \sum_{(g)} (-1)^{\dim \mathcal{Y}^{(g)}} \int_{[T\mathcal{Y}^{(g)}]} ch (u_g)_+ \cdot td\Bigl(T\mathcal{Y}^{(g)}\otimes \mathbb{C}\Bigr).
\]
Note that $u_g$ is already a supertrace, so the supertraces act as the identity on \(u_g\).

Each term on the right-hand side is (up to the sign \((-1)^{\dim \mathcal{Y}^{(g)}}\)) the principal term (corresponding to the non-twisted stratum \(g_1\in \langle g\rangle\), the group generated by \(g\)) of Kawasaki's index formula for \(u_g\) over \(\mathcal{Y}^{(g)}\).
By Kawasaki's index formula,
\[
\ind [u_g]_+ = \sum_{(g_1)} (-1)^{\dim \mathcal{Y}^{(g)(g_1)}} \int_{[T\mathcal{Y}^{(g)(g_1)}]} ch \!\left(\frac{i_{g_1}^\ast u_g}{str_{g_1}\, \wedge^\ast (\mathcal{N}^\ast_{g_1}\otimes \mathbb{C})}\right)_+ \cdot td\Bigl(T\mathcal{Y}^{(g)(g_1)}\otimes \mathbb{C}\Bigr).
\]
Note that all terms with \(\mathcal{Y}^{(g)(g_1)}=\mathcal{Y}^{(g)}\) in the sum are the same. We move the terms corresponding to \(g_1\notin \langle g\rangle\) to the left and divide both sides by \(m(g)\) to obtain
\[
\begin{array}{ll}
&\dint_{[T\mathcal{Y}^{(g)(g_1)}]} ch \!\left(\dfrac{i_{g_1}^\ast u_g}{str_{g_1}\, \wedge^\ast (\mathcal{N}^\ast_{g_1}\otimes \mathbb{C})}\right)_+ \cdot td\Bigl(T\mathcal{Y}^{(g)(g_1)}\otimes \mathbb{C}\Bigr) \\[1mm]
= \dfrac{1}{m(g)}\ind [u_g]_+ - &\dsum_{(g,g_1)\in \mathcal{G}} (-1)^{\dim \mathcal{Y}^{(g)(g_1)}} \dint_{[T\mathcal{Y}^{(g)(g_1)}]} ch \!\left(\dfrac{i_{g_1}^\ast u_g}{str_{g_1}\, \wedge^\ast (\mathcal{N}^\ast_{g_1}\otimes \mathbb{C})}\right)_+ \cdot td\Bigl(T\mathcal{Y}^{(g)(g_1)}\otimes \mathbb{C}\Bigr).
\end{array}
\]
This equation expresses the desired integration on the sector \(\mathcal{Y}^{(g)}\) in terms of the index of \(u_g\) and integrations of the same form on 2-sectors. A similar result with \((\mathcal{Y}^{(\mathbf{g})},\, \mathcal{Y}^{(\mathbf{g})(g_{n+1})},\, u_{\mathbf{g}})\) in place of \((\mathcal{Y}^{(g)},\, \mathcal{Y}^{(g)(g_1)},\, u_g)\) expresses the relevant integrations on \((n+1)\)-sectors in terms of \(\ind [u_\mathbf{g}]_+\) and integrations on \((n+2)\)-sectors. 

Next, we have
\[
(i_\mathbf{g})_! (u_\mathbf{g})|_{\mathcal{N}_\mathbf{g}} = str_g\, u|_{\mathcal{N}_\mathbf{g}}.
\]
This also follows by induction on \(|\mathbf{g}|\); the base case and the induction step both reduce to
\[
(i_g)_! u_{(g)}|_{\mathcal{N}_g} = str_g\, u|_{\mathcal{N}_g}
\]
for appropriate \(\mathcal{N}\) and \(u\). Indeed, the latter is equivalent to
\[
\pi_g^\ast i_g^\ast \Bigl(str_g\, \frac{u}{\wedge^\ast (\mathcal{N}_g^\ast\otimes \mathbb{C})}\Bigr) \cdot str_g\, \wedge^\ast (\mathcal{N}_g^\ast\otimes \mathbb{C}) = str_g\, u,
\]
which can be rewritten as
\[
\Bigl(\pi_g^\ast i_g^\ast \Bigl(str_g\, \frac{u}{\wedge^\ast (\mathcal{N}_g^\ast\otimes \mathbb{C})}\Bigr) - str_g\, \frac{u}{\wedge^\ast (\mathcal{N}_g^\ast\otimes \mathbb{C})}\Bigr) \cdot str_g\, \wedge^\ast (\mathcal{N}_g^\ast\otimes \mathbb{C}) = 0.
\]
Here, the first term vanishes on \(\mathcal{Y}^g\) and the second term vanishes on \(\mathcal{N}^g\setminus \mathcal{Y}^g\). 

Finally, applying part (1) of Proposition 3.3, we obtain
\[
\begin{array}{ll}
\ind ([u_\mathbf{g}]_+) &= p_\ast^{fake}\Bigl(\mu (u_\mathbf{g})\Bigr)\\[1mm]
& = p_\ast^{fake}\Bigl(i_{\mathbf{g}}^!\, \mu\bigl((i_{\mathbf{g}})_!(u_{\mathbf{g}})\bigr)\Bigr)\\[1mm]
& = p_\ast^{fake}\Bigl(i_{\mathbf{g}}^!\, \bigl(\mu (u)\bigr)^{(\mathbf{g})}\Bigr),
\end{array}
\]
and the proposition follows.
\end{proof}

It also follows from the definitions of \(i^!\) and \(p_\ast\) that for any inclusion \(i: \mathcal{X} \to \mathcal{Y}\), each fake push-forward to a point is respected by the wrong-way pull-back \(i^!\). Thus, push-forwards to points are preserved by \(i^!\).

\subsection{Compatibility}
\

Given an element \(\alpha \in K^0(\mathcal{Y})\) and an element \(T \in K_0(\mathcal{Y})\), we define their pairing by
\[
T(\alpha) = p_\ast (\alpha \cap T).
\]
The following proposition is a consequence of Propositions 3.2 and 3.4.

\begin{prop}
For \(\beta \in K^0(B^\ast \mathcal{Y}, \partial B^\ast \mathcal{Y})\) and \(\alpha \in K^0(\mathcal{Y}, \partial \mathcal{Y})\), we have
\[
(\mu(\beta))(\alpha) = p_\ast (\beta \boxtimes \alpha).
\]
\end{prop}

\section{K-theoretic virtual fundamental cycle}
\

\begin{Def}
A stably almost complex global Kuranishi chart consists of a quadruple \(A = (G, Y, E, s)\), where \(G\) is a compact Lie group, \(Y\) is a locally \(C^1\) \(G\)-manifold with an almost free \(G\)-action (so that \(T^\ast_G Y\) is stably complex), \(E\) is a locally \(C^1\) \(G\)-equivariant stably complex vector bundle over \(Y\), and \(s\) is an equivariant section \(Y \to E\).
\end{Def}

In \cite{Abouzaid23}, Abouzaid, Mclean, and Smith constructed stably almost complex global Kuranishi charts on the compactified moduli space of stable maps \(\bar{\mathcal{M}}_{g.n}(X,d)\) to general symplectic targets \(X\). We shall make the following modifications and use orbifold Kuranishi charts.

\begin{Def}
An almost complex global orbifold Kuranishi chart consists of a triple \(\mathcal{A} = (\mathcal{Y, E},\sigma)\), where \(\mathcal{Y}\) is a smooth, compact orbifold with boundary, \(\mathcal{E}\) is a complex smooth orbi-bundle over \(\mathcal{Y}\), and \(\sigma\) is a continuous section \(\mathcal{Y} \to \mathcal{E}\) that vanishes nowhere on \(\partial \mathcal{Y}\).
\end{Def}

Note that almost complex global orbifold Kuranishi charts exist. Since the real dimensions of \(Y\) and the fibers of \(E\) have the same parity, we can always add the same \(\mathbb{R}^k\) to them (and add the identity map \(id^k: \mathbb{R}^k \to \mathbb{R}^k\) to the section \(s\)) to obtain complex \(Y\) and \(E\). An almost complex global Kuranishi chart \((G, Y, E, s)\) gives an almost complex orbifold Kuranishi chart \((\mathcal{Y, E},\sigma)\) by taking the quotient by the \(G\)-action and then shrinking to a compact neighborhood of \(\sigma^{-1}(0)\). Since \(C^1\) structures induce smooth structures, we can also modify \(Y\) and \(E\) to smooth orbifolds and smooth bundles.
Note that \(s\) is not smooth in general.

By \cite{Abouzaid23} Proposition 4.62 and Lemma 3.4 (see also \cite{Abouzaid21}), global Kuranishi charts \(A = (G, Y, E, s)\) for the same moduli space of stable maps of a certain type are connected by a sequence of operations of the following types (or their inverses):

\begin{enumerate}
\item
\textbf{Free quotients:} This operation replaces \((G, Y, E, s)\) with \((G/H, Y/H, E/H, s)\), where \(E/H\) is viewed as a \(G/H\)-bundle over \(Y/H\).

\item
\textbf{Induced maps:} This replaces \((G, Y, E, s)\) with \((G', Y\times_G G', \pi^\ast E, \pi^\ast s)\), where \(\pi: Y \times_G G' \to Y\) is the projection induced by a slicing \(G' \to G\).

\item
\textbf{Stabilization:} Stabilization by a \(G\)-equivariant vector bundle \(\pi: W \to Y\) replaces \((G, Y, E, s)\) with \((G, W, \pi^\ast E \oplus \pi^\ast W, \pi^\ast s \oplus \Delta)\), where \(\Delta\) is the tautological “diagonal” section of \(W \to \pi^\ast W\).

\item
\textbf{Open embeddings:} This replaces \((G, Y', E|_{Y'}, s|_{Y'})\) with \((G, Y, E, s)\) when \(Y' \subseteq Y\) is an open set containing \(s^{-1}(0)\).
\end{enumerate}

In turn, the orbifold Kuranishi charts are related by a sequence of the following operation (and its inverse):

\noindent (*) Replace \(\mathcal{A} = (\mathcal{Y, E}, \sigma)\) by \(\mathcal{A}' = (\mathcal{W}, \mathcal{N} \oplus \mathcal{E}, \sigma \oplus \delta)\), where \(\pi: \mathcal{N} \to \mathcal{Y}\) is a complex orbi-bundle, \(\mathcal{W} \subseteq \mathcal{N}\) is a closed neighborhood of \(\sigma^{-1}(0)\), and \(\delta\) is the restriction of the diagonal map \(\mathcal{N} \to \pi^\ast \mathcal{N}\) to \(\mathcal{W}\).

\subsection{Virtual fundamental cycle of global Kuranishi chart}
\

We define the K-theoretic virtual fundamental cycle \(\mathcal{O}^{vir}_{\mathcal{Y,E},\sigma} \in K_0(\mathcal{Y})\) of an almost complex global orbifold Kuranishi chart \((\mathcal{Y,E},\sigma)\) as follows. Let \(\pi: T^\ast \mathcal{Y} \to \mathcal{Y}\) be the projection that forgets the cotangent vector. Let
\[
\mathcal{S} = \wedge^\ast_\mathbb{C} T^\ast \mathcal{Y},\quad \theta = \wedge^\ast_\mathbb{C} \mathcal{E},\quad \text{and}\quad \mathcal{O}= \mathcal{S} \otimes \theta.
\]
Note that \(\mathcal{S}\) and \(\theta\) are \(\mathbb{Z}_2\)-graded bundles (and hence so is the tensor product \(\mathcal{O}\)).
We add Hermitian structures to \(\mathcal{Y}\) and \(\mathcal{E}\), and let
\[
\begin{array}{ll}
p &= \lambda_\xi \otimes 1 + \varepsilon\, 1 \otimes \pi^\ast \lambda_\sigma: \pi^\ast \mathcal{O} \to \pi^\ast \mathcal{O},\\[1mm]
\mathcal{O}^{vir}_{\mathcal{Y,E},\sigma} &= \mu(p).
\end{array}
\]
Recall that \(\varepsilon\) is the operator on \(\mathbb{Z}_2\)-graded objects acting as \((-1)^a\) on the direct summand of degree \(a\).
Also, recall that \(\lambda\) is defined as follows: given a section \(\tau: \mathcal{X} \to \mathcal{F}\) of a Hermitian bundle, we define \(\lambda_\tau: \wedge^\ast \mathcal{F} \to \wedge^\ast \mathcal{F}\) by
\[
\lambda_\tau(\omega) = \tau(x) \wedge \omega + i_{\tau(x)}(\omega)
\]
for \(x \in \mathcal{X}\), where \(i_{\tau(x)}\) denotes contraction with \(\tau(x)\) via the Hermitian metric on \(\mathcal{F}\). Form \(\lambda_{\xi}: \pi^\ast \mathcal{S} \to \pi^\ast \mathcal{S}\) and \(\lambda_\sigma: \mathcal{E} \to \mathcal{E}\) corresponding to the diagonal map \(\xi: T^\ast \mathcal{Y} \to \pi^\ast T^\ast \mathcal{Y}\) and to the section \(\sigma: \mathcal{Y} \to \mathcal{E}\). Observe that \(\lambda_\xi\) (respectively, \(\lambda_\sigma\)) is of degree \(1\), so \(p\) is also of degree \(1\). Moreover, \(\lambda_\xi\) and \(\lambda_\sigma\) are self-dual and square to \(|\xi|^2 \cdot 1\) and \(|\sigma(x)|^2 \cdot 1\), respectively; hence,
\[
p^\ast p = (|\xi|^2 + |\sigma(x)|^2) \cdot 1.
\]
Thus, \(p\) is invertible outside the compact set
\[
\mathcal{M} = \{x \in \mathcal{Y}: |\sigma(x)| = 0 \text{ and } |\xi| = 0\},
\]
and therefore represents an element in \(K^0(T^\ast \mathcal{Y})\).

It follows from Proposition 3.3 that, in terms of the K-theoretic virtual fundamental cycle, the operation (*) corresponds to first pushing forward by \(i_\ast\) (defined in Section 2.3) and then pulling back by \(j^!\) (defined in Section 3.3).

\begin{cthm}{1} (For Kuranishi charts)
If \(\sigma\) is transversal to the zero section \(0\), then
\[
\mathcal{O}^{vir}_{\mathcal{Y,E},\sigma} = i_\ast [\sigma^{-1}(0)].
\]
\end{cthm}

\begin{proof}
Note that the space \(\mathcal{M} = \sigma^{-1}(0)\) is stably almost complex. We can define its K-theoretic fundamental cycle by Bott periodicity, since \(K_0(\Sigma^k \mathcal{M}) = K_0(\mathcal{M})\). This identification is exactly our \(i^!\) for the inclusion \(i: \mathcal{M} \to B(\mathcal{M} \times \mathbb{R}^{2k})\). Next, by part (2) of Proposition 3.3, the push-forwards \(i_\ast\) and the wrong-way pull-back \(i^!\) are intertwined by the choose-an-operator map, thus proving the theorem.
\end{proof}

The pairing \(\langle \cdot,\cdot \rangle\) between a \(K\)-homology class and a \(K\)-cohomology class is defined as the push-forward \(p_\ast\) (defined in Section 3.6) of their cap product (defined in Section 3.1). Invariants defined via the virtual fundamental cycle in \cite{Abouzaid23} (see equation (3.26)) are identified with those defined by our K-theoretic virtual fundamental cycle, as a direct consequence of Proposition 3.5.

\begin{cthm}{2} (For Kuranishi charts)
\[
\langle \mathcal{O}^{vir}_\sigma,\mathcal{V}\rangle = \ind [\mathcal{E} \otimes \pi^\ast \mathcal{V}]_+.
\]
\end{cthm}

\subsection{Virtual fundamental cycle of virtual orbifolds}
\

We further define the virtual fundamental cycle as
\[
i^! (\mathcal{O}^{vir}_\sigma) \in K_0(\mathcal{M}).
\]

\begin{prop}
This virtual fundamental cycle is independent of the global Kuranishi chart.
\end{prop}

\begin{proof}
This follows immediately from the fact that \(i^! j^! = (ji)^!\).
\end{proof}

We define \(p_\ast\) by the same formula as in Definition 3.1, with the “normal bundles” replaced by
\[
\mathcal{N}_{\mathbf{g}}^{\mathcal{M}} = \mathcal{N}_\mathbf{g} \ominus \mathcal{E}_\mathbf{g}|_{\mathcal{M}^\mathbf{g}}
\]
in an arbitrary global chart. Then we immediately have the following proposition.

\begin{prop}
\(\mathcal{N}_{\mathbf{g}}^{\mathcal{M}}\) is independent of the global Kuranishi chart \((\mathcal{Y,E},\sigma)\).
\end{prop}

We also have

\begin{cthm}{2} (For orbispaces with almost complex global Kuranishi chart)
Let \(i: \mathcal{M} \to \mathcal{Y}\) be the inclusion. Then
\[
\langle \mathcal{O}^{vir}_\sigma, V\rangle = \langle \mathcal{O}^{vir}_\mathcal{M}, i^\ast V\rangle.
\]
\end{cthm}

\begin{proof}
We only need to prove that the index \(p_\ast^{fake}\) on each \(K\)-homology element \(e_\mathbf{g}^{-1} i_\mathbf{g}^! \varphi^{(g)}\) coincides. Note that this index depends only on the \(\mathbb{Z}_2\)-graded Hilbert space \(H\) and the operator \(T\) involved, and is independent of the module structure \(\phi\) over the ring of continuous functions. To obtain the same (virtual) Hilbert spaces and Fredholm operators, for each \(\mathbf{g}\) we can consider 
\[
\frac{e_\mathbf{g}^{-1} i_\mathbf{g}^! \varphi^{(g)}}{\wedge^\ast (\mathcal{E}|_{\mathcal{Y}^{\mathbf{g}}} / \mathcal{E}^{\mathbf{g}})^\ast}
\]
over \(\mathcal{Y}^{\mathbf{g}}\) and \(\mathcal{E}|_{\mathcal{Y}^{\mathbf{g}}} / \mathcal{E}^{\mathbf{g}}\). The coincidence of indices now follows from the end of Section 3.5.
\end{proof}

\begin{cor}
The invariants are independent of the global Kuranishi chart chosen, and agrees with that defined in \cite{Abouzaid23}.
\end{cor}

\author{Dun Tang, Department of Mathematics, University of California, Berkeley, 
1006 Evans Hall, University Drive, Berkeley CA94720.
E-mail address: dun\_tang@math.berkeley.edu}

\end{document}